\documentclass{gtart}

\newcommand{\pathtotrunk}{./}

\usepackage{amsmath,amssymb,amsfonts}
\usepackage{ifpdf}
\usepackage{ulem}

\ifpdf
\usepackage[pdftex,all,color]{xy}
\else
\usepackage[all,color]{xy}
\fi

\SelectTips{cm}{}

\ifpdf
\newcommand{\pathtodiagrams}{\pathtotrunk diagrams/pdf/}
\else
\newcommand{\pathtodiagrams}{\pathtotrunk diagrams/eps/}
\fi

\newcommand{\mathfig}[2]{{\hspace{-3pt}\begin{array}{c}%
  \raisebox{-2.5pt}{\reallyincludegraphics[width=#1\textwidth]{\pathtodiagrams #2}}%
\end{array}\hspace{-3pt}}}

\newcommand{\arxiv}[1]{\href{http://arxiv.org/abs/#1}{\tt arXiv:\nolinkurl{#1}}}
\newcommand{\doi}[1]{\href{http://dx.doi.org/#1}{{\tt DOI:#1}}}

\newcommand{\googlebooks}[1]{(preview at \href{http://books.google.com/books?id=#1}{google books})}

\def\RCS$#1: #2 ${\expandafter\def\csname RCS#1\endcsname{#2}}\RCS$Revision: 107 $ \RCS$Date: 2009-06-08 16:15:51 -0700 (Mon, 08 Jun 2009) $

\theoremstyle{plain}
\newtheorem{prop}{Proposition}[section]
\newtheorem{proposition}{Proposition}[section]
\newtheorem{conj}[prop]{Conjecture}
\newtheorem{conjecture}[prop]{Conjecture}
\newtheorem{thm}[prop]{Theorem}
\newtheorem{lem}[prop]{Lemma}

\newtheorem*{cor*}{Corollary}

\newtheorem*{example}{Example}
\newtheorem*{defn*}{Definition}             

\newenvironment{rem}{\noindent\textsl{Remark.}}{}  
\newenvironment{remark}{\noindent\textsl{Remark.}}{}  
\numberwithin{equation}{section}

\newcounter{comment}
\newcommand{\noop}[1]{}

\newcommand{\package}[1]{\code{#1`}}
\newcommand{\code}[1]{{\tt #1}}

\def\clap#1{\hbox to 0pt{\hss#1\hss}}

\newcommand{\Natural}{\mathbb N}
\newcommand{\Integer}{\mathbb Z}
\newcommand{\Z}{\mathbb Z}

\newcommand{\Real}{\mathbb R}
\newcommand{\R}{\mathbb R}

\renewcommand{\imath}{\mathfrak{i}}
\renewcommand{\jmath}{\mathfrak{j}}

\newcommand{\bbar}{\overline}

\newcommand{\KhH}[1]{\operatorname{Kh}\left(#1\right)}
\newcommand{\Kh}[2]{\KhH{#1}\left(#2\right)}
\newcommand{\rank}{\operatorname{rank}}

\newcommand{\MMA}{\code{Mathematica}{} }

\newcommand{\abs}[1]{\left|#1\right|}

\newcommand{\eset}{\emptyset}

\ifpdf
\usepackage[pdftex,plainpages=false,hypertexnames=false,pdfpagelabels]{hyperref}
\else
\usepackage[dvips,plainpages=false,hypertexnames=false]{hyperref}
\fi

\usepackage{breakurl}

\ifpdf
\usepackage[pdftex]{graphicx}
\else
\usepackage[dvips]{graphicx}
\fi

\usepackage{etex}
\usepackage{dcpic}
\usepackage{pictex}
\usepackage{pinlabel}
\usepackage{color}

\newtheorem{fact}[prop]{Fact}
\newtheorem{theorem}[prop]{Theorem}

\usepackage{microtype}

\title{Man and machine thinking about the smooth $4$-dimensional Poincar\'{e} conjecture.}

\author{Michael~Freedman}
\address{
}%
\email{michaelf@microsoft.com}

\author{Robert~Gompf}
\address{
}%
\email{gompf@math.utexas.edu}

\author{Scott~Morrison}
\address{
}%
\email{scott@tqft.net}

\author{Kevin~Walker}
\address{
}%
\email{kevin@canyon23.net}

\date{
  First edition: the mysterious future
  This edition: \today.
}


\begin{document}

\begin{abstract}
While topologists have had possession of possible counterexamples to the smooth $4$-dimensional Poincar\'{e} conjecture (SPC4) for over 30 years, until recently no invariant has existed which could potentially distinguish these examples from the standard $4$-sphere. Rasmussen's $s$-invariant, a slice obstruction within the general framework of Khovanov homology, changes this state of affairs. We studied a class of knots $K$ for which nonzero $s(K)$ would yield a counterexample to SPC4. Computations are extremely costly and we had only completed two tests for those $K$, with the computations showing that $s$ was $0$, when a landmark posting of Akbulut \cite{0907.0136} altered the terrain.  His posting, appearing only six days after our initial posting, proved that the family of ``Cappell--Shaneson'' homotopy spheres that we had geared up to study were in fact all standard.  The method we describe remains viable but will have to be applied to other examples.  Akbulut's work makes SPC4 seem more plausible, and in another section of this paper we explain that SPC4 is equivalent to an appropriate generalization of Property R (``in $S^3$, only an unknot can yield $S^1 \times S^2$ under surgery''). We hope that this observation, and the rich relations between Property R and ideas such as taut foliations, contact geometry, and Heegaard Floer homology, will encourage $3$-manifold topologists to look at SPC4.
\end{abstract}

\maketitle


\setcounter{tocdepth}{2}
\tableofcontents
\section{Introduction} \label{section_intro}

The smooth 4-dimensional Poincar{\'{e} conjecture (SPC4) is the ``last man standing'' among the great problems of classical geometric topology.\footnote{The 3-dimensional Schoenflies problem (Does a smooth 3-sphere imbedded in $\R^4$ bound a smooth 4-ball?) is a special case.  It asks if there are invertible homotopy spheres.}  Manifold topology has been transformed by contact with physics and geometry so much that few of the questions studied today would have been recognizable fifty years ago; SPC4 is the exception.  The ``conjecture'' can be formulated as:

\begin{conjecture}[SPC4] \label{conjecture_SPC4}
A smooth four dimensional manifold $\Sigma$ homeomorphic to the 4-sphere $S^4$ is actually diffeomorphic to it, $\Sigma = S^4$.
\end{conjecture}

The opinion of researchers in this area tends to lean in the direction that the conjecture is {\em false}.  There are three reasons for this.  First, Donaldson theory and Seiberg-Witten theory produce a myriad of examples of multiple smooth structures on closed simply connected 4-manifolds (although all of these have second homology).  Second, there are several constructions which give potential counterexamples.  Only a small subset of these have been ``killed,'' that is, proved standard.  Third, the best known tool for constructing diffeomorphisms, the $h$-cobordism theorem, is broken, that is it fails for smooth $h$-cobordisms between 4-manifolds.\footnote{Perhaps in light of recent developments, one should say that Akbulut \cite{0907.0136} and even more recent work \cite{0908.1914} of the second author suggest that a direct handle-by-handle approach or possibly a combinatoric approach organized by the concept of ``broken Lefschetz fibrations'' might be fruitful.}

In the other direction, there are two strands of research that could provide a positive argument (that is, that the conjecture is {\em true}).  The first derives from a theorem of Gromov \cite[\S 0.3.C]{MR809718} which states that any symplectic structure on $\Sigma - \text{pt.}$ which is standard near the deleted point is actually symplectomorphic to $dp_1dq_1 + dp_2dq_2$ on $T^*\Real^2 = \Real^4$.  Perhaps punctured homotopy $4$-spheres can be given a symplectic structure standard at infinity. Because the existence of a symplectic structure is \emph{not} always preserved by homemorphism, the argument would have to be special to homotopy spheres.

A second possible avenue will be explained in section \ref{section_propR} of this paper.  We show that SPC4 is implied by an appropriate generalization of Gabai's Property R theorem \cite[Corollary 8.3]{MR910018}:
\begin{quote}
``In $S^3$, only an unknot can yield $S^1 \times S^2$ under surgery.''
\end{quote}
It should not be a huge surprise that some generalization of Property R implies SPC4, since if $\Sigma$ has a handle decomposition consisting of exactly one $k$-handle for each index $k=0,2,3,4$ and no handles of index $1$, then $\Sigma = S^4$.  To see this, let $N$ = $\partial(0,2\text{-handles})$ be the boundary after the $0$- and $2$-handles are attached.  In order to attach a $3$-handle (to the essential $2$-cycle), it is necessary that $N$ have the form $N= S^1 \times S^2 \# \Sigma^3$.  But for $\partial(0,2,3 \text{-handles})$ to be $S^3$ so that the $4$-handle can be attached, we see $N=S^1 \times S^2$.  By Property R, the $2$-handle must attach to the unknot, i.e. in a standard way.  Similarly, up to isotopy there is a unique $2$-sphere in $N$ for the $3$-handle to attach to. The core, $\text{pt.} \times S^2$, and co-core, $S^1 \times \text{pt.}$ meet in one point so the Morse Lemma cancels the $2$- and $3$-handles.  The result is the standard Morse function on $\Sigma$, showing $\Sigma = S^4$.  In fact, a Property R-like statement can be tailored to be equivalent to SPC4.  We hope that this observation will open SPC4 to 3-manifold specialists.

But let us suppose, again, that SPC4 is false.  One would expect a proof of this to calculate some invariant sensitive to smooth structure.  The Donaldson polynomial and their Seiberg-Witten analogs seem ill suited to homotopy spheres since they address how families of self-dual connections (or harmonic spinors obeying a quadratic constraint) specialize to $2$-cycles in the $4$-manifold.  We report here the results of a model large scale computer calculation of a different invariant which -- with the proper example in hand -- may have the power to distinguish certain homotopy $4$-spheres from the standard $4$-sphere.

We studied two homotopy $4$-spheres $\Sigma$ produced by Cappell and Shaneson in 1978 \cite{MR0413117} and extracted a local problem: the {\em slice problem} for certain knots $K$, of whether $K$ bounds an imbedded disk in $\frac{1}{2}\Real^4$. This can be answered affirmatively if $\Sigma = S^4$. We thus shifted attention from $4$-manifold invariants to knot invariants which give an obstruction to the slice problem. Our proposed invariant is exactly Rasmussen's \cite{math.GT/0402131} $s$-invariant associated to the Khovanov homology  \cite{MR1740682, MR1917056, MR2174270} of a knot $K$.  Half the absolute value $\frac{1}{2} |s(K)|$ of the $s$-invariant serves as a lower bound for the genus of smoothly embedded surfaces in $\frac{1}{2} \Real^4$ bounding $K \subset \Real^3$. In the usual formulation, the $s$-invariant is extra information beyond the Khovanov homology, and comes from the spectral sequence relating Khovanov homology and Lee's \cite{MR2173845} variation thereon. For some knots, however, including our $K$, it is possible to extract the $s$-invariant directly form the Khovanov homology.
The invariant $|s(K)|$ is an obstruction to the slice problem, but 766 hours of computer calculation show $s(K)=0$, yielding no information.  Akbulut's recent posting \cite{0907.0136} eliminates the best known subfamily of CS-homotopy spheres (indexed by the integers), by proving them to be standard. This is precisely the family we had geared up to compute, incidentally confirming our calculation that $s(K)=0$. At the end of section \ref{section_history} we give more information regarding how our method could be applied to other potential counterexamples to SPC4.

Localizing the problem to ``slicing a knot $K$'' is necessitated by the limits of present knowledge.  Perhaps current work on `blob homology' by two of us, Morrison and Walker, will allow a Khovanov homology skein module $A(\Sigma^4)$ to be defined, computed, and compared to $A(S^4)$.  Even if possible, exponential scaling will make the brute force study of examples nearly impossible.  Theoretical tools will be needed.  Since we do localize to a slice problem, the reader may wonder if the known localizations of gauge theory to slice obstructions, in particular the $\tau$-invariant, could be used in place of $s$. (The $\tau$-invariant also gives a lower bound, of $\frac{1}{2} \abs{\tau(K)}$, and when the $s$-invariant was discovered it was conjectured to coincide with $\tau$, on the basis of small examples, although this was later proved false in \cite{MR2405163}.) The answer however is no.  Gauge theory invariants will not see the difference between slicing in an exotic versus standard 4-ball.  The gauge theoretic lower bounds to $4$-ball genus are localizations of adjunction formulas (relative to special classes) and these, by standard neck stretch arguments are insensitive to the smooth structure near a single point. In particular, if $K$ is slice in any homotopy $4$-ball, then $\tau(K)=0$.

Khovanov homology is newer and less well understood than gauge theory.  The only definitions we know involve the coordinate structure of space -- in three dimensions, projections are used to define the homology groups and $s$-invariant, and in four dimensions, ``movies'' of projections establish what the $s$-invariant measures.  Thus, it is possible that Khovanov homology has the power to distinguish an exotic $4$-ball (which will have no smooth radius function with 3-sphere levels) from the standard one, a possibility too tempting to overlook. Although the first computer calculation yielded $s=0$, it may be possible to produce additional, computationally feasible, test cases.  If there are general reasons why Khovanov homology is insensitive to the smooth structure on a ball, they are presently unknown.

The authors would like to thank Dror Bar-Natan, Nathan Dunfield, Rob Kirby and Martin Scharlemann for interesting and useful conversations relating to this paper. Robert Gompf was partially supported by NSF grant DMS-0603958 during this work, and Michael Freedman, Scott Morrison and Kevin Walker were at Microsoft Station Q.

\section{Historical background and potential counterexamples} \label{section_history}

In 1976, both Cappell and Shaneson \cite{MR0413117} and Gordon \cite{MR0440561} studied ways to produce homotopy spheres $\Sigma$ by a two step process.  Start with a closed $3$-manifold $M$ and a self-diffeomorphism $\phi$ to produce a mapping torus $$\frac{M \times I}{\phi (M \times 1) \sim M \times 0}.$$  Then surger a cross-section circle; this often produces a homotopy sphere.  The two possible circle framings yield homotopy spheres which are related by the Gluck construction (described below) on the dual $2$-sphere.  Since the original interest was on producing a distinct pair of $2$-knots with identical complement, some effort was expended, particularly in \cite{MR0440561}, to achieve a stronger theorem by ensuring that $\Sigma$ was recognizably diffeomorphic to $S^4$, not just a homotopy sphere.

But quickly the emphasis shifted to the more exciting possibility that some of the homotopy spheres produced in this fashion might be exotic. The Cappell-Shaneson examples come from suitable self-diffeomorphisms $\phi$ of the $3$-torus. Up to obvious equivalences, there are finitely many such $\phi$ for each trace (and only one for each trace between $-4$ and 9 \cite{MR780575}). Thus most research has focused on the representative family $\Sigma_m$ determined by the matrices $$ A_m = \left| \begin{array}{ccc} 0 & 1 & 0 \\ 0 & 1 & 1 \\ 1 & 0 & m+1 \end{array} \right|, \qquad m \in \Z,$$ and the ``harder'' choice of framed $1$-surgery. (Aitchison and Rubenstein \cite{MR780575} showed that the other choice of framing yields $S^4$.) The study of such examples by explicit handle descriptions was initiated by Akbulut and Kirby \cite{MR816520}, who produced an elegant diagram of $\Sigma_0$ without $3$-handles.  Subsequently, one of us \cite{MR1081936} was able to show that $\Sigma_0 = S^4$ by introducing a $(2,3)$-handle pair into the Morse function.  (This $(2,3)$ pair has significance for generalizing Property R which we will discuss in section \ref{section_propR}.) The manifolds $\Sigma_m$, $m \neq 0$, remained mysterious until proven standard in \cite{0907.0136}.

It should be mentioned that ten or twenty years earlier, two other sources of homotopy spheres were known. At the end of this section we comment on adapting our approach to the study of these examples.  First, given a balanced presentation $P$, e.g. $\langle x,y\; | \; xyx=yxy, \; x^4=y^5 \rangle$, for the trivial group, let $C$ be the corresponding $2$-complex.  Embed $C$ into $\Real^5$ and let $\mathcal{N}(C)$ be a regular neighborhood.  The boundary $\partial \mathcal{N} = \Sigma$ is a homotopy sphere, uniquely determined by $P$.  If $P$ is Andrews-Curtis equivalent to the empty presentation, it is easy to see that $\Sigma = S^4$, otherwise the situation is unclear. However, various apparently nontrivial presentations, such as the one above, are now known to generate the standard 4-sphere, as a consequence of study of the Cappell-Shaneson examples \cite{MR1081936}, \cite{0907.0136}. The second source of examples is the {\em Gluck construction} \cite{MR0146807}. For any knotted 2-sphere in $S^4$, remove its tubular neighborhood $S^2\times D^2$ and reglue it by the nontrivial but homologically trivial diffeomorphism of $S^2\times S^1$, rotating $S^2$ once as we travel around $S^1$. The result is a homotopy sphere that is not known to be $S^4$ in general. However, it is known to be trivial for many 2-knots, such as twist-spun knots  \cite{MR0440561}, \cite{MR508892}, doubles of ribbon disks \cite{MR1707327}, and various 2-knots arising from the Cappell-Shaneson construction (those trivialized in  \cite{MR1081936},  \cite{0907.0136} and \cite{0908.1914}).

Returning to the family $\Sigma_m$, for thirty years these examples patiently awaited the development of an invariant that {\em might} separate them from $S^4$.  (They are now known to be standard \cite{0907.0136}, but we nevertheless explain our general approach in their context for it was in $\Sigma_{-1}$ and $\Sigma_1$ that we did our computations.)  As explained in the introduction, gauge theory was not up to the task. Recently, the Khovanov homology  link invariant \cite{MR1740682, MR1917056, MR2174270}, a categorification of the Jones polynomial, has presented itself as a possibility in the form of Rasmussen's $s$-invariant \cite{math.GT/0402131}.  The reason it is possible that this is the right tool to separate standard from exotic $4$-spheres (or equivalently $4$-balls) is that the definitions we have for $Kh$ and $s$ are coordinate intensive.  Projections are used to define $Kh$ and $s$ and ``movies'' of projections are used to prove the properties of $s$.  It should be admitted at the outset that a less combinatorial, more abstract formulation of $Kh$ -- perhaps analogous to Witten's \cite{MR990772} reformulation of the Jones polynomial -- might establish general properties of $Kh$ that would make it useless in detecting homotopy spheres.  However, nothing of this sort is known so we feel the problem of testing homotopy spheres against computer calculations of $s$ is irresistible.

The $s$-invariant is a lower bound to the $4$-ball genus of a knot $K$,
\begin{equation}\label{eqn:s(K)}
\frac{1}{2} |s(K)| \leq \text{genus}_4(K)
\end{equation}
Let us explain how we localize the problem by showing that the condition $\Sigma_1 = S^4$ implies, for a certain knot $K: S^1 \hookrightarrow S^3$, that $K$ is slice, i.e. has $\text{genus}_4(K)=0$.  The handle structure (drawn in detail in section \ref{sec:bands}) of $\Sigma_1$ consists of:  0-handle $\cup$ two (1-handles) $\cup$ two (2-handles) $\cup$ (4-handle).  Since there are no exotic diffeomorphisms of $\partial(4\text{-handle})=S^3$, we may without loss of generality pull off the $4$-handle and ask if the remaining homotopy ball $B_1$ (with boundary equal to $S^3$) is standard, that is ``is $B_1 = B^4$?''  Recall some terminology: a four dimensional $k$-handle is a pair $(B^k \times B^{4-k}, \partial B^k \times B^{4-k})$ which is to be glued onto the boundary of the union of lower index handles along $\partial B^k \times B^{4-k}$.  $B^k \times 0$ is called the ``core'' and $0\times B^{4-k}$ is called the ``co-core.''

The co-cores of the two $2$-handles in $B_1$ are a pair of disjoint disks bounding a two component link $L \subset \partial B_1 = S^3$.  These disks show that $L$ is ``slice'' in $B_1$.  But if $B_1 \neq B^4$ then there is a possibility that $L$ is {\em not} a slice link in the conventional sense of bounding disjoint imbedded disks in $B^4$.  It is this possibility that we study via the $s$ invariant.

Now $L$ is a two component link and the $s$ invariant is defined for knots; is this a problem?  Actually, there is a generalization of $s$ to links\footnote{See \cite{MR2462446} for an integer valued invariant; there is also a stronger version, unpublished. Neither is immediately computable by the method described in \S \ref{sec:calculations}.} so the strongest approach would be to compute the ``generalized $s$'' of $L$.  Unfortunately, $L$ is well beyond the reach of any computer (unless a better algorithm is discovered).  The picture (Figures \ref{fig:gompf-1} and \ref{fig:link}) we draw for $L$ has 222 crossings and (more importantly) $\text{girth} \approx 24$.

Any knot $K_b$ obtained from $L$ by joining its two components by a rectangular band $b$ will bound an embedded disk in $B_1$ (the band sum of the pair of disks with boundary $L$).  In this sense, $K_b$ is slice in $B_1$ but again might fail to be slice in $B^4$.  In a sense, calculating $s(K_b)$, for any band $b$, provides information: if some $s(K_b) \neq 0$, then $B_1 \neq B^4$ as $K_b$ is slice in $B_1$ but not in $B^4$.  Although information may be lost in passing from $L$ to $K_b$, our only hope seems to be to seek out bands so that $K_b$ is small enough to compute $s$.

In the end we found three promising bands $b_1$, $b_2$, and $b_3$ so that the resulting knots $K_1$, $K_2$, and $K_3$ satisfy:
\begin{center}
\begin{tabular}{|c||c|c|c|}
  \cline{1-4}
  \; & apparent crossing $\#$ & apparent girth & comment \\
  \cline{1-4}
  $K_1$ & 67 & 14 & ribbon \\
  \cline{1-4}
  $K_2$ & 78 & 14 & \; \\
  \cline{1-4}
  $K_3$ & 86 & 16 & \; \\
  \cline{1-4}
\end{tabular}
\end{center}

The first knot $K_1$ is useless to us since ribbon knots are slice.  We ultimately computed (766 hours on a dual core AMD Opteron 285 with $32$gb of RAM) that $s(K_2)=0$, so unfortunately it too failed to distinguish $B_1$ from $B^4$.  The last knot $K_3$ appears to require an order of magnitude more space and time to compute.  The calculation was discontinued when Akbulut proved $B_m = B^4$ for all $m$. Shortly thereafter, Nathan Dunfield contacted us with a better presentation of $K_3$, see \S \ref{sec:K3}, which would probably make the calculation (indeed, of the full Khovanov homology, not just the $s$-invariant) feasible.

Our other, shorter calculation, done in $B_{-1}$ is briefly discussed in section \ref{sec:minus1}.

To summarize, suppose $B'$, $\partial B' = S^3$, is a homotopy 4-ball consisting of a 0-handle, $k$ 1-handles and $k$ 2-handles.  Suppose $K \subset S^3$ is the result of band summing the $k$ cocore circles into a single knot:

\begin{fact} \label{fact_1}
If $s(K) \neq 0$ then $B' \neq B^4$ and the SPC4 is false.
\end{fact}

\begin{fact} \label{fact_2}
If $s(K) \neq 0$, then the Andrews-Curtis conjecture is false for the $k$ generator, $k$ relation presentation $P$ of the trivial group given by the handle decomposition of $B'$.
\end{fact}

\begin{proof}
The handle structure of $B'$ defines the presentation $P$ of the trivial group.  The handle structure on $B'$ stabilizes to a handle structure on $B' \times I$ also with $k$ 1-handles and $k$ 2-handles, giving the identical presentation $P$.  If $P$ is Andrews-Curtis equivalent to the empty presentation, then there is no {\em geometric} obstruction in 5-dimensions to covering the AC moves with handle slides (and births/deaths of (1-handle, 2-handle) pairs).  This means that $B' \times I = B^5$.  But then $$\text{Double}(B') = (B' \cup D^4) \# (\bbar{B'} \cup D^4) = \partial (B' \times I) = \partial B^5 = S^4.$$

This implies $B' \# (\bbar{B'} \cup D^4) = B^4$.  Since $K$ bounds a disk $\Delta$ imbedded in $B'$, which we may assume misses the region of the connected sum, $K$ also bounds $\Delta' \subset B' \# (\bbar{B'} \cup D^4) = B^4$.  But $s(K) \neq 0$ implies that there is no imbedded disk $\Delta' \subset B^4$ with $\partial \Delta' = K$.  This contradiction shows that $P$ cannot be AC equivalent to the empty presentation.
\end{proof}

The examples $\Sigma_m$ are now all known to be trivial \cite{0907.0136}, and the second author has found a simpler and more conceptual proof that an even larger family of Cappell-Shaneson homotopy 4-spheres is standard \cite{0908.1914}. The latter method may eventually show that all Cappell-Shaneson spheres are standard, raising the possibility that SPC4 may actually be true. In any case, none of the remaining  Cappell-Shaneson spheres are currently known to admit handle presentations without 3-handles. There is one remaining source of interesting homotopy spheres without 3-handles known to the authors, originating in Melvin's thesis \cite{melvin-thesis}: If a knotted 2-sphere in $S^4$ admits a height function with at most two local minima, then such an example arises by Gluck construction. (A simpler proof from \cite{MR1094545} also appears as Exercise~6.2.12(d) of \cite{MR1707327}, along with an algorithm for constructing the dual link diagram without 1-handles.) As discussed previously, many knots are known only to produce $S^4$ by Gluck construction, so to proceed in this direction, one should first find a clever way to generate 2-knots that seem likely to be interesting. On the other hand, our method can in principle be used in the presence of 3-handles, by locating a knot bounding a surface of some genus in a homotopy ball, then trying to prove that this genus cannot be realized in $B^4$.

\paragraph{Summary:} The existence of any non-slice knot (or link) $L$ in the boundary 3-sphere of a homotopy ball with an equal number of 1 and 2-handles, $L$ built from the 2-handle co-core boundaries by attaching a forest (disjoint collection of trees) of bands implies the failure of both SPC4 and AC.  Khovanov homology (or its variants categorifying the $\mathfrak{sl}_3$ polynomial \cite{MR2100691,MR2336253, math.GT/0612754} or $\mathfrak{sl}_n$  polynomials \cite{MR2391017,math.QA/0505056, math.GT/0612406}) may hold promise for establishing that a link $L$ is {\em not slice}.  More generally, if 3-handles are present in the handle decomposition there is still a vast array of knots and links which are seen to bound a system of surfaces of known genus. Inequalities such as \ref{eqn:s(K)} can then, in principle, show that the ambient space holding these surfaces is not the standard 4-ball.  On the other hand, the very steep (super exponential?) escalation of computational costs, both in time and space, for computing such invariants may limit such explorations.  Because computation may be limited to girth $\leq 14$ and \#~crossings $\leq 90$, the correct strategy may be to simplify the boundary knot (or link) with additional bands.  That is, to {\em not} necessarily seek to produce the boundary $K$ of a surface in a homotopy ball of the smallest possible genus.  The disadvantage of such an approach is that to detect a homotopy ball one would have to find $|s(K)| >> 0$; the advantage is that with additional bands, many knots could be generated in a size range, say girth $\leq 12$, crossings $\leq 60$, where each calculation of $s$ could be done in tens of seconds.

\section{Generalizing Property R} \label{section_propR}

Property R is now identified with Gabai's theorem \cite{MR910018}.

\begin{theorem}[Property R] \label{theorem_propR}
If surgery on $K\subset S^3$ is $S^1 \times S^2$ then $K$ is the unknot.
\end{theorem}

Gabai actually proved that any spherical class in the surgered manifold would imply that $K$ is trivial.  This theorem has had tremendous importance in Floer theory and contact geometry, and has close relations to Property P (if surgery on $K$ yields $S^3$ then $K$ unknotted).  Its proof showcased taut foliation and sutured manifold techniques and is central to 3-manifold topology.  We hope these powerful 3-dimensional methods might be imported to study the SPC4.  In fact, this section will describe some possible generalizations of Property R, one of which is equivalent to SPC4.  Unfortunately, these generalizations do not statically discuss a single knot or link but rather address the consequences of some set of ``moves,'' so a direct generalization of Gabai's proof does not seem likely.

To relate Property R to SPC4, we need to recall a 4-dimensional version of Kirby's ``calculus'' theorem \cite{MR0467753} (see also \cite{MR1707327}).  First, some notation:  A connected 4-manifold $M$ has a handle decomposition with one 0-handle whose boundary is the ``blackboard.''  The 1-handles are drawn (assuming orientability and following Akbulut) as an unlink with a dot on each component; 0-surgery on this $p$-component unlink is the boundary $S_p=\# _{p \text{ copies }} S^1 \times S^2$ of $(\text{0-handle})\cup p (\text{1-handles})$. (By convention, $S_0=S^3$.) The 2-handles are drawn as a framed link in the complement of the dotted unlink. Each framing is represented by an integer, namely the linking number of the component with its pushoff via the framing (which is also the coefficient of the resulting surgery on the boundary 3-manifold, measured in the background $S^3$). Assuming $M$ is closed, the 3- and 4-handles comprise a regular neighborhood of a wedge of $q$ circles in $M$ that attaches in essentially a unique way to the boundary $S_q$ of the union of 0-, 1-, and 2-handles. Thus, we need not bother to keep track of the $q$ 3-handles and 4-handle.

\begin{theorem} \label{thm_kirby}
Let $M$ and $M'$ be 4-manifolds given as above by links $L$ and $L'$, respectively, with each link component dotted or framed. If $M$ and $M'$ are diffeomorphic (preserving orientations) then $L$ and $L'$ are related by compositions of the moves in Figure~\ref{kirbymoves}.
\begin{figure}[!htpb]
\labellist \Large\hair 2pt

  \pinlabel $\eset$ at 80 100
  \pinlabel $\eset$ at 80 165
  \pinlabel $(1)$ at -60 165
  \pinlabel $(2)$ at -60 100
  \pinlabel $(3)a,b,c$ at -60 35

\endlabellist
\centering
\includegraphics[scale=0.6]{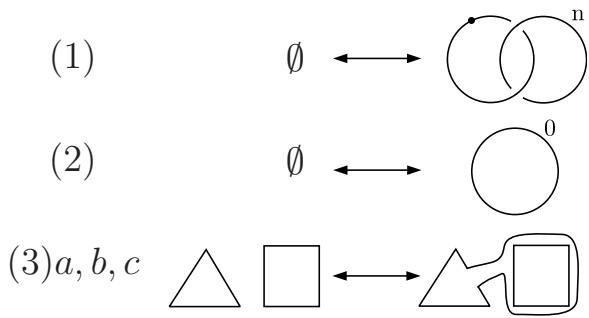}
\caption{(1) birth/death of a 1-handle, 2-handle pair, (2) birth/death of a 2-handle, 3-handle pair (3-handle omitted), (3) handle slides.} \label{kirbymoves}
\end{figure}
Here, $n$ can be any integer, so the framing in (1) is arbitrary. However, we can require $n$ to be zero (so the framing is tangent to the obvious spanning disk) if (a) both $L$ and $L'$ have a component with odd framing or (b) all framings are even and $H_1(M;\Z _2)=0$ (or more generally if the given diffeomorphism preserves the induced spin structures).

\end{theorem}

The three versions, $a,b,c$ of (3) deserve comment.  1-handles are often represented (cf. \cite{MR1707327}) by pairs of deleted balls whose boundaries are to be glued across a mirror.  The separating sphere $S$ has one hemisphere realized as the disk spanning \rlap{$^\bullet$}$\bigcirc$ in our notation and the other is present after the longitude is filled in by surgery (Figure~\ref{reflglue}).

\begin{figure}[htpb]
\noop{
\labellist \small\hair 2pt

  \pinlabel $S$ at 25 85
  \pinlabel $\text{glue by reflection}$ at 280 60

\endlabellist
}
\centering
\includegraphics[scale=0.7]{\pathtodiagrams john/reflglue}
\caption{} \label{reflglue}
\end{figure}

The obvious rule for sliding 1-handles on the deleted ball representation becomes band summing of one dotted circle with a parallel copy of the other, where the band is disjoint from a fixed collection of spanning disks for the dotted circles.  This is case $(a)$ of 1-handle slides.  (Exercise: the dotted circle representation of 1-handle slides is contravariant.)  Case $(b)$ is sliding a 2-handle (undotted circle) over a dotted circle.  Here, there is no restriction on the band.  The slide is nothing but an isotopy of the attaching region of the 2-handle in the boundary after the 1-handles are attached. (To compare the two notations, it is helpful to keep track of the arc along which we are to bring the two balls together to make a dotted unknot. Then these slides occur as attaching circles isotope through the connecting arc.) In Figure \ref{slide} we show an example in both notations.

\begin{figure}[!htpb]
\centering
\includegraphics[scale=0.85]{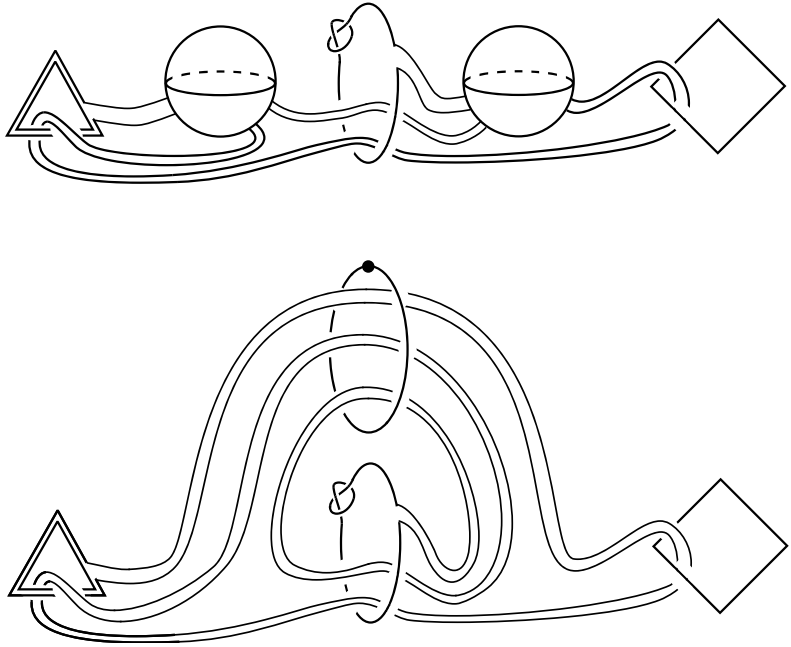}
\caption{A handle slide of the trefoil over the triangle is shown in the two different notations.} \label{slide}
\end{figure}

Finally, $(c)$ is the familiar sliding of one 2-handle over another. Recall that when we slide one component over another, the former changes its framing coefficient to the sum of the two coefficients plus twice the linking number of the components (oriented so that the two parallel curves point in the same direction). The same rule applies in $3(b)$ where the dotted circle has coefficient 0.

\begin{proof}[Proof of Theorem \ref{thm_kirby}.]
 This is almost implicit in \cite{MR0467753}. By Cerf theory, the two handle decompositions are related by handle moves, and 0- and 4-handles need not be created. The remaining moves translate into (1) -- (3) (arbitrary $n$). Only the parity of $n$ is significant, since it changes by 2 when the framed circle is slid over the dotted circle in (1) (move $3(b)$). Similarly, moves $3(a,b)$ allow us to assume the framed circle in (1) is knotted and linked with other components  (i.e., the introduction of ``circumcision pairs'' \cite{MR1456165} is redundant in this setting, see Figure \ref{circpairs}). If $L$ has an odd-framed component $K$, we can now change the parity of $n$ in move (1) by sliding the $n$-framed unknot over $K$ and simplifying. If both $L$ and $L'$ have odd-framed components, it now suffices to take $n=0$ in each move (1). (Use the above procedure to generate move (1) with $n=1$ in both diagrams $L$ and $L'$. We now have an extra Hopf link with an odd framing throughout the computation, so we can generate odd-parity moves (1) from the $n=0$ case as needed.) If all framings of $L$ and $L'$ are even, then the diagrams determine spin structures on $M$ and $M'$. If $H_1(M;\Z _2)=0$ then the spin structure on $M$ is unique, so the given equivalence preserves spin structures. This latter condition allows us to translate the Cerf theory into diagrams respecting the spin structure, so all framings are even throughout. (Note that the condition is necessary since odd framings cannot be created from an even diagram without an odd move (1).)
 \end{proof}

\begin{figure}[!htpb]
\noop{
\labellist \small\hair 2pt

  \pinlabel $\text{In (1)}$ at -20 75
  \pinlabel $\implies$ at 165 75

\endlabellist
}
\centering
\includegraphics[scale=0.7]{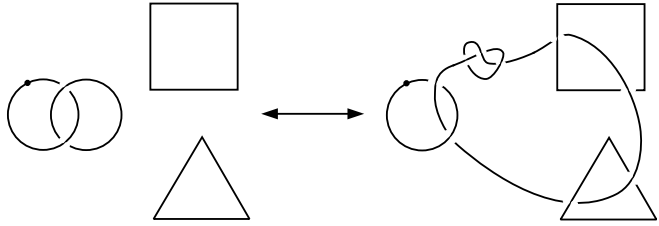}
\caption{Circumcision pairs \cite{MR1456165} are redundant in the presence of move  $3(b)$.} \label{circpairs}
\end{figure}

We can now restate SPC4 in a form generalizing Property R.

\begin{conjecture} \label{conj_1'}
Let $L=L_1\cup L_2$ be a link in $S^3$ with $L_1$ a dotted $p$-component unlink and $L_2$ a framed link of $p+q$ components. Suppose that $L_2$ normally generates $\pi_1(S^3-L_1)$, and that surgery on $L$ (with dotted components 0-framed) is diffeomorphic to $S_q$. Then (a) there is a sequence of moves (1), (2), and (3) as above transforming $L$ to the empty diagram, and (b) if all framings on $L_2$ are even, we can require $n=0$ in each move (1).
\end{conjecture}

Note that for $p=0$ and $q=1$, Property R is precisely the assertion that a single move $(2)^{-1}$ suffices. When $p=0$, $L$ must be a $q$-component link with all framings and linking numbers zero. In general, the linking matrix of $L$ presents a homomorphism with cokernel $H_1(S_q)=\Z^q$, and the diagonal elements (framings) can be arranged by (1) -- (3) to be all even (since the corresponding homotopy 4-sphere has $H^2(\Sigma;\Z_2)=0$ and so admits a spin structure). Note that it is essential to restrict the number of 2-handles. Fewer than $p+q$ cannot yield a closed, simply connected 4-manifold (which must have Euler characteristic $\ge 2$), whereas a sufficient excess would allow as a connected summand an exotic, closed, 1-connected 4-manifold, so no recognition theorem could then be possible. (For a weaker statement in this case, see Conjecture \ref{1-handle}.)  Also note that while our hypothesis involves simple connectivity, it does so in the benign context where a new trivial relation (move (2)) can be added at will.  Thus, we would not expect to meet subtle presentation issues such as the Andrews-Curtis Conjecture in analyzing the scope of the moves.

\begin{proposition} \label{prop_conj1_equiv_conj1'}
Both (a) and (b) of Conjecture \ref{conj_1'} are equivalent to SPC4.
\end{proposition}

\begin{proof}
Given $L$ as in the conjecture, it is easy to verify that the corresponding closed 4-manifold is simply connected with Euler characteristic 2, so it is a homotopy 4-sphere. Thus, SPC4 implies it is diffeomorphic to $S^4$. Theorem \ref{thm_kirby} now implies $L$ can be transformed to the empty link in both cases (a) and (b). Conversely, any homotopy 4-sphere has a handle decomposition given by a diagram as in the conjecture, so the latter ((a) or (b)) implies the sphere is standard.
\end{proof}

\begin{remark} \label{remark_conj1'_proof}
There is evidence that moves (1) and (2) (in the direction increasing the number of link components) cannot be dispensed with, although there is currently no proof. More precisely, it seems unlikely that the condition $pq=0$ can always be preserved when simplifying a link as in the conjecture. This evidence arises from the handle decomposition of the homotopy 4-sphere $\Sigma_0$ introduced in \cite{MR816520} and mentioned in the previous section. Failure of the Andrews-Curtis Conjecture for the presentation $\langle x,y\; | \; xyx=yxy, \; x^4=y^5 \rangle$ would imply that that handle decomposition cannot be trivialized without introducing a (2,3)-handle pair, i.e. a move (2) with $p=2$, $q=0$, increasing $q$ to 1. The explicit trivialization in \cite{MR1081936} via such a pair can also be reinterpreted as a move (1), with $n=0$, $p=0$, $q=2$ (increasing $p$ to 1), that is likely to be indispensable. The resulting trivialization of a 2-component link is exhibited explicitly in \cite{property2R} in the context of property 2R, a generalization of property R to two-component links in $S^3$.
\end{remark}

A natural approach to SPC4 is to study related conjectures. We can try to prove weaker conjectures such as Conjecture \ref{conj_1'} with the added hypothesis that $p=0$ and possibly a bound on $q$, e.g. $q\le 2$. As we have already seen, it is unlikely that we can always simplify to the unlink while preserving these hypotheses, but one might still hope to make progress by increasing $p$ and $q$ with moves (1) and (2). See \cite{0901.2319, property2R} for further discussion. In the other direction, we can try to disprove stronger conjectures. We close this section with two such candidates.

First we consider the relation of {\em stable diffeomorphism}, that is, diffeomorphism up to connected sum with sufficiently many copies of  $S^2 \times S^2$. Such a connected sum is given in Kirby calculus by Figure \ref{move4}.

\begin{figure}[htpb]
\labellist \Large\hair 2pt

  \pinlabel $\eset$ at -25 29

\endlabellist
\centering
\includegraphics[scale=0.7]{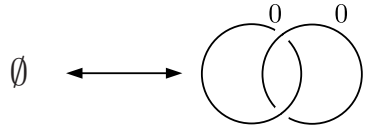}
\caption{The extra move in the Kirby calculus, generating stable diffeomorphism.} \label{move4}
\end{figure}

\noindent Thus, adding this fourth link move to our other three generates the relation of stable diffeomorphism of the associated 4-manifolds. A traditional question, on which there has been little progress, asks when a stable diffeomorphism can be destabilized to an ordinary diffeomorphism. Since closed, simply connected 4-manifolds are stably diffeomorphic (via the $h$-cobordism theorem) whenever they are homotopy equivalent, we do know that the link in Conjecture \ref{conj_1'} can be reduced to the empty link via moves of type (1), (2), (3), and (4).  Conjecture \ref{conj_1'} says that use of type (4) can be avoided under the given hypotheses. We can now weaken our $\pi_1$ assumption.  Probably due to merely a lack of techniques, we know of no pairs of integral homology 4-spheres which are stably diffeomorphic but not diffeomorphic.  This means that we could make a stronger Conjecture \ref{conj2} (probably dubious), equivalent to the assertion that for homology $4$-spheres, stable diffeomorphism implies diffeomorphism.

\begin{conjecture} \label{conj2}
Let $L=L_1\cup L_2$ be as in Conjecture \ref{conj_1'}, but with $L_2$ only given to generate $H_1(S^3-L_1;\Z)$ (rather than $\pi_1(S^3-L_1)$). Suppose $L'$ is obtained from $L$ by moves (1) -- (4) (and their inverses), with the same number of occurrences of move (4) and its inverse. Then $L'$ can also be obtained from $L$ by moves (1) -- (3).
\end{conjecture}

Since integral homology 4-spheres admit unique spin structures, the conjecture can be interpreted in two equivalent ways: We can allow arbitrary $n$ in move (1), or require $n=0$ and all framing coefficients even. If $L_2$ normally generates $\pi_1(S^3-L_1)$, we know from above that $L$ can be transformed to the empty link by moves (1) -- (4), so Conjecture \ref{conj2} implies Conjecture \ref{conj_1'}. Since closed, oriented 4-manifolds are stably diffeomorphic whenever they are homeomorphic \cite{MR769285}, Conjecture \ref{conj2} has the (unlikely?) consequence that homology 4-spheres cannot admit exotic smooth structures. The conjecture clearly becomes false if we allow $L_2$ to have more than $p+q$ components, even when $L_2$ normally generates $\pi_1(S^3-L_1)$, since many simply connected 4-manifolds (including large connected sums of $S^2 \times S^2$) do admit exotic smooth structures. In particular, there are links that are transformable to the empty link by moves (1) -- (4), but not if we disallow move (4) in the direction increasing the number of link components.

Our other strengthening of SPC4 involves odd-index handles. A manifold built without 1-handles is obviously simply connected, and a closed 4-manifold without 3-handles can be alternatively presented without 1-handles (by reversing the sign of the Morse function). It is an old open question whether every simply connected, closed 4-manifold has a handle decomposition without odd-index handles. Such a decomposition would have $b_2(M)$ 2-handles, so an affirmative answer would imply SPC4 and (via Property P) nonexistence of exotic smooth structures on $\mathbb{CP}^2$. An affirmative answer has always seemed unlikely, due to an abundance of potential counterexamples. However, the best known example has recently been shown to admit such a handle structure \cite{0805.1524}. This, together with recent progress on Cappell-Shaneson spheres \cite{0907.0136,0908.1914}, seems to increase the likelihood that such handle structures always exist. We translate the assertion into a conjecture about links:

\begin{conjecture} \label{1-handle}
Let $L=L_1\cup L_2$ be as in Conjecture \ref{conj_1'}, but with no restriction on the number of components of $L_2$. If $p+q\ne 0$ then $L$ can be transformed by moves (1) -- (3) to a similar link with a smaller value of $p+q$.
\end{conjecture}

This is easily seen to be equivalent to the conjecture that all closed, simply connected 4-manifolds have decompositions without odd-index handles. Unlike previous conjectures, this one becomes weaker if we add the restriction that all framings must be even -- the previous statement becomes restricted to spin manifolds. Either way, it is easy to see directly that the conjecture implies Conjecture \ref{conj_1'}. 
\section{Bands and isotopies}
\label{sec:bands}
Recall that in \cite[Figure 17]{MR1094545} the second author gave a handle presentation of the Cappell-Shaneson spheres $\Sigma_m$ with no $3$-handles. This is reproduced as Figure \ref{fig:gompf-fig-17} here.

\begin{figure}[!ht]
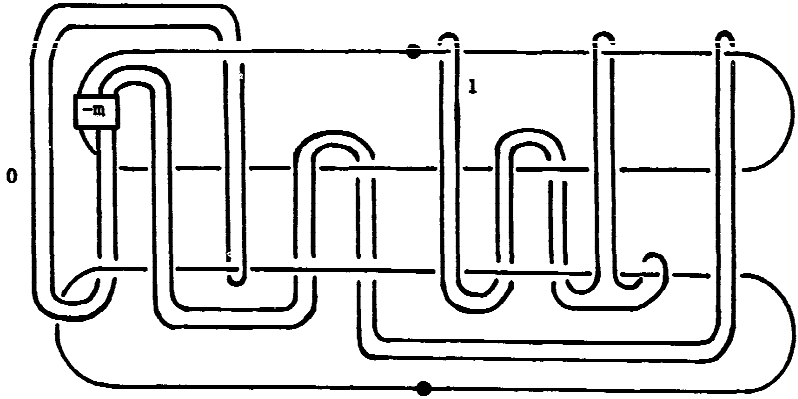

\begin{equation*}
\mathfig{0.8}{gompf-fig-17}
\end{equation*}
\caption{Figure 17 from \cite{MR1094545}, showing the handle presentation of the Cappell-Shaneson sphere $\Sigma_m$.}%
\label{fig:gompf-fig-17}
\end{figure}

We want to interpret this as a complicated handle presentation of the standard boundary $3$-sphere (after we've removed the $4$-handle), and perform a sequence of Kirby moves turning this into the trivial presentation. As we do this, we need to follow along the two meridian loops around the $2$-handles; this link in the $3$-sphere is the one we hope is not slice. This link can be read off from Figure 9 of \cite{MR1094545}, reproduced
in Figure \ref{fig:gompf-m-original}. Ignoring the component labelled by $xy$ there, the remaining two component link is what we're after, and we'll call that $L_m$ throughout. 
Note that the unknotted circle appearing in that diagram as a dashed line is not a third component, but notation for a full
positive twist.

\begin{figure}[!ht]
\begin{equation*}
\mathfig{0.8}{gompf-fig-9}
\end{equation*}
\caption{Figure 9 from \cite{MR1094545}, showing the two component cocore link $L_m$. What appears to be a third, unknotted, component drawn with a dashed line is actually
notation for a full positive twist on the strands passing through it.}%
\label{fig:gompf-m-original}
\end{figure}

We'll first simplify this picture of $L_m$, making the twist circle lie flat. In our pictures, the twist circle in shown in blue,
to help distinguish it from the actual link, which appears in black and green. At each step, the part of the knot we're
about to move is marked with dashes (and red, if you're reading this in color), and often its destination is indicated with a thin dashed line.

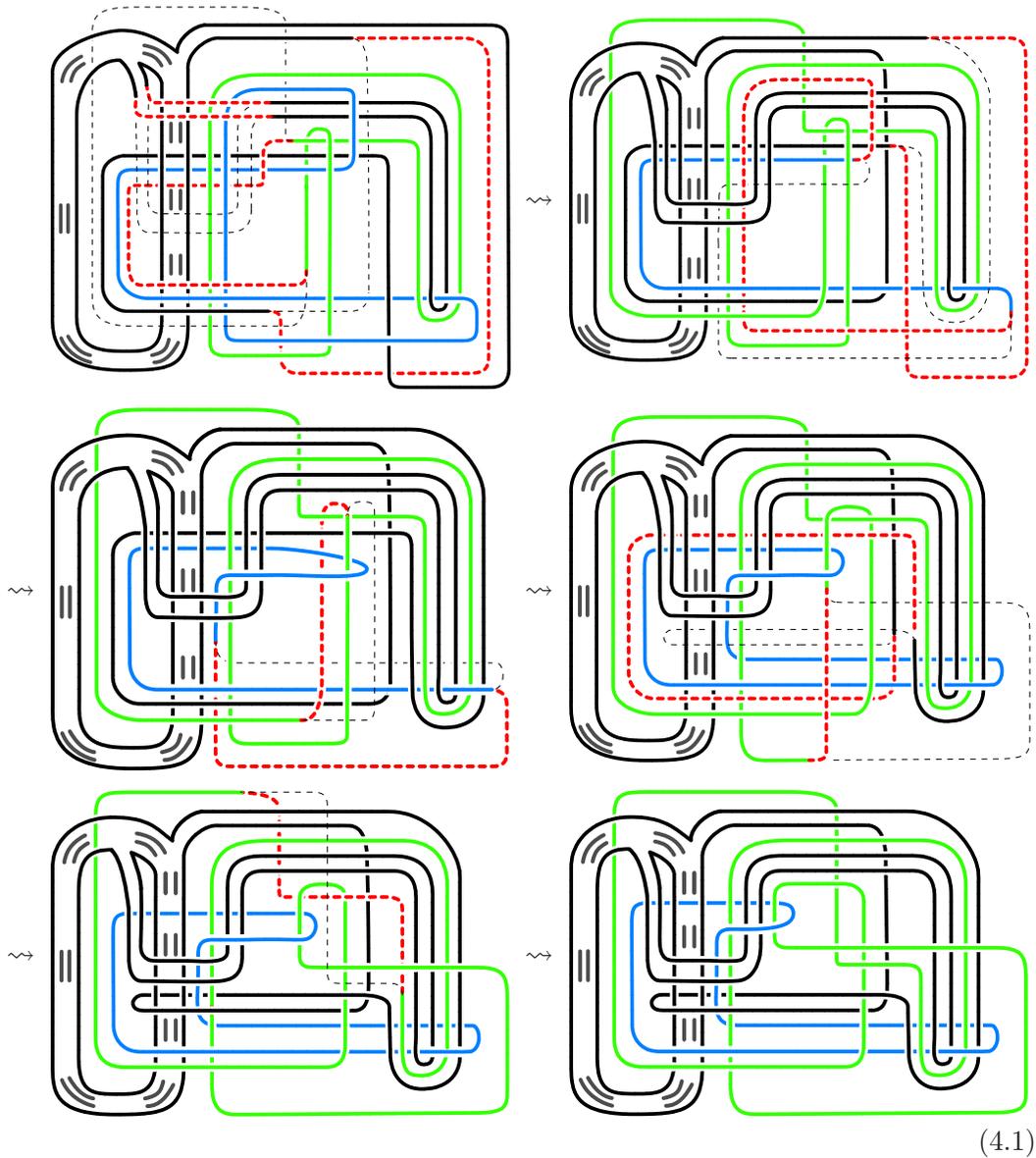
\begin{align}
                 \mathfig{0.45}{gompf/1} & \rightsquigarrow \mathfig{0.45}{gompf/2} \notag \displaybreak[1] \\ 
\rightsquigarrow \mathfig{0.45}{gompf/3} & \rightsquigarrow \mathfig{0.45}{gompf/4} \notag \displaybreak[1] \\
\rightsquigarrow \mathfig{0.45}{gompf/5} & \rightsquigarrow \mathfig{0.45}{gompf/6} \label{eq:Lm-untwisted}
\end{align}

Finally, specializing to $m=1$, we obtain the link $L_1$ in Figure \ref{fig:gompf-1}.

\begin{figure}[!ht]
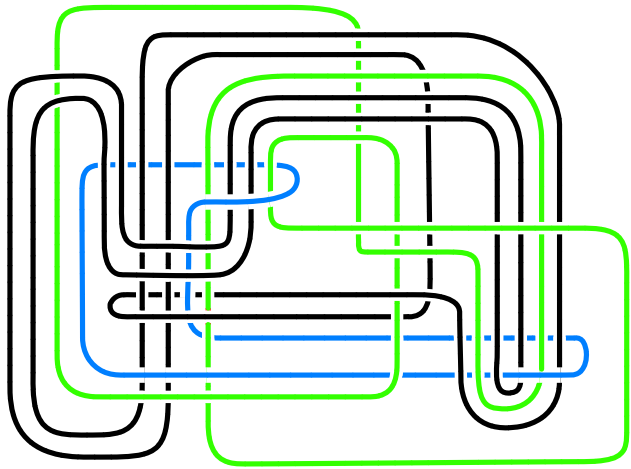

\begin{equation*}
\mathfig{0.6}{gompf-1}
\end{equation*}
\caption{The two component cocore link with $m=1$.}%
\label{fig:gompf-1}
\end{figure}

It's difficult to work directly with these pictures; we'll first perform a series of isotopies to
ensure that all the strands passing through the `twist circle' are
parallel. Omitting several steps, we obtain the link in Figure \ref{fig:link}. Note that we've chosen an arbitrary orientation, both because we're about to write down the Gauss code, which needs an orientation, and because Khovanov homology is an invariant of oriented links (although only depending weakly on the orientation, like the Jones polynomial).
\begin{figure}[!ht]
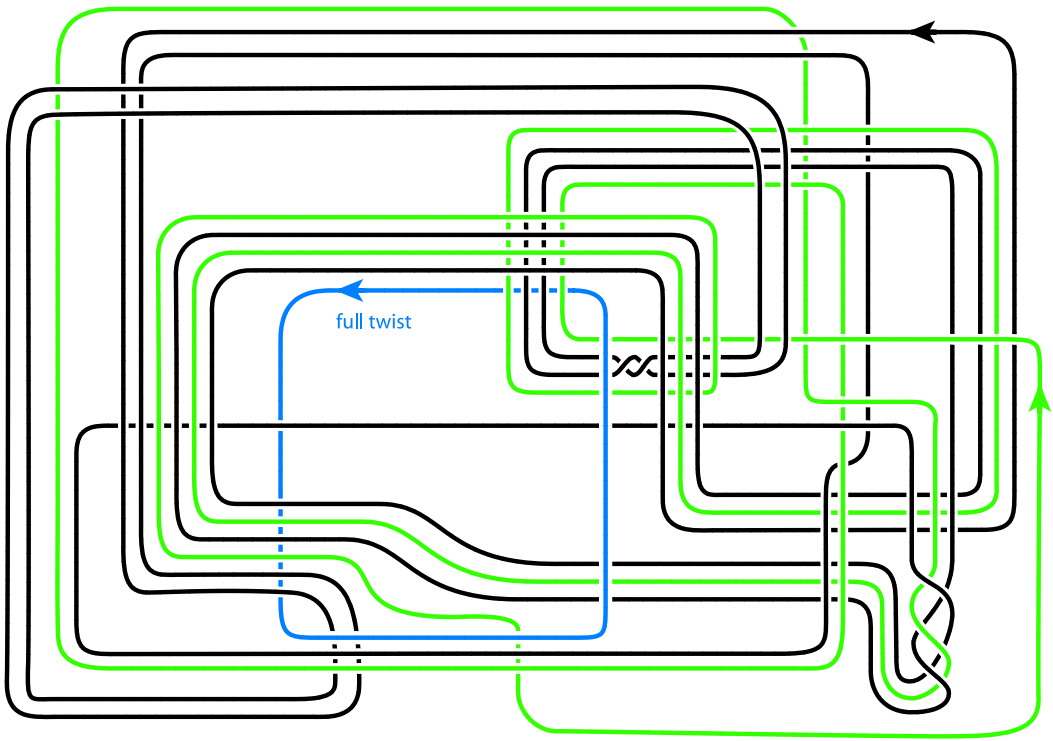

\begin{equation*}
\mathfig{1.0}{link}
\end{equation*}
\caption{Another picture of the two component cocore link with $m=1$. Now all the strands passing through the twist circle are parallel.}%
\label{fig:link}
\end{figure}

This presentation has $222$ crossings. ($112$ are outside the twist circle, and $11 \times 10 = 110$ are in the full twist.) It has Gauss code (first the green component, then the black, in each case starting at the orientation arrow):
\begin{quote}
\tt
(1, 2, 3, 4, 5, 6, 7, -8, -9, -10, -11, -12, -13, -14, 15, -16, 17,
-18, 19, -20, 21, -22, 23, -24, 25, -26, 27, -28, 29, -30, 31, -32,
33, -34, -35, -36, -37, -38, -39, 40, -6, -41, -42, 43, 44, 45, 46,
47, 48, 49, 50, 51, 52, -53, -54, -55, -56, -57, -58, -59, -40, -7,
41, 60, 61, 62, 63, -64, 65, 66, -67, -48, -68, -69, 70, -71, 72,
-73, 74, -75, 76, 77, 78, 79, 80, 81, 22, -23, -82, -83, -84, -85,
-86, 87, 88, 89, 90, 35, 12, 91, 92, 93, 94, -95, -45, -96, -62, -97,
-2, 98, 57, -99, -100, -101, -102, -88, -103, -104, -105, -29, 28,
106, 107, 108, -109, 110, -111, 84, -79, 112, -113, 114, -115, 116,
-117, 118, -119, -120, -93, -121, 122, 123, 10, 37, 124, 125, 101,
126, -127, -74, -128, -129, -114, -130, -131, -19, 18, 132, 133, 115,
134, 135, 73, 136, 137, -138, 139, -140, -141, -50),

(55, -142, -143,
144, -145, -139, -146, -70, -147, -148, -118, -149, -150, -15, 14,
151, 152, 119, 153, 154, 69, 155, 140, 156, -157, -52, 53, 143, 158,
100, 159, 160, 38, 9, -123, -161, -91, -162, 163, -164, -151, 150,
-165, 166, -132, 131, -167, 168, -81, 82, -169, 170, -171, 172, -106,
105, -173, 174, 32, -33, -175, -90, -176, -124, -160, -177, 59, 178,
-4, 179, 97, 180, -181, -65, 182, 67, -49, -183, -155, 146, -184,
185, -136, 127, 75, 186, 187, 113, 188, 167, 20, -21, -168, -189,
-112, -190, -191, -76, 192, 102, 193, 176, 36, 11, 161, 194, 121,
195, -196, -44, -197, -61, -179, -3, 198, 58, -199, -159, -125, -193,
-89, -200, -174, -31, 30, 173, 104, -107, 201, -202, 203, -204, 83,
-80, 189, -188, 130, -133, 205, -206, 149, -152, 164, -163, -207,
-92, -194, -122, 8, 39, 177, 199, 99, 208, 142, 54, -51, -209, -156,
145, 138, 184, 71, 210, 211, 117, 206, 165, 16, -17, -166, -205,
-116, -212, -213, -72, -185, -137, 214, -158, -208, 56, -98, -198,
-178, -5, -60, -215, -43, 196, 95, 216, 217, 68, 183, 141, 209, 157,
-144, -214, -126, -192, -87, -218, -108, -201, -172, -27, 26, 171,
202, 109, 219, -220, 85, -78, 190, -187, 129, -134, 212, -211, 148,
-153, -221, -94, -195, 42, 215, 197, 96, 222, 64, 181, -66, -182,
-47, -217, -154, 147, -210, 213, -135, 128, -186, 191, -77, 86, 220,
111, 204, 169, 24, -25, -170, -203, -110, -219, 218, 103, 200, 175,
34, 13, 162, 207, 120, 221, -216, -46, -222, -63, -180, -1)
\end{quote}

We never seriously considered giving it directly to our Khovanov homology computer programs.
(We've found an ordering of the crossings in this
presentation that results in a girth of 24. It might be possible to do better, but probably not by much. See section \S \ref{sec:algorithm} for a discussion of girth as an obstacle to Khovanov homology calculations.)

Instead, we decided to look for some bands that can be added to the link, hoping to form a much simpler knot. If the original two component
link is slice, of course every such knot must be slice too, and so an obstruction for any knot obtained by adding a band will do.
Of course, we might be throwing the baby out with the bath water here!

As mentioned in the introduction, we considered three different bands on $L_1$, resulting in knots $K_1$, $K_2$ and $K_3$:
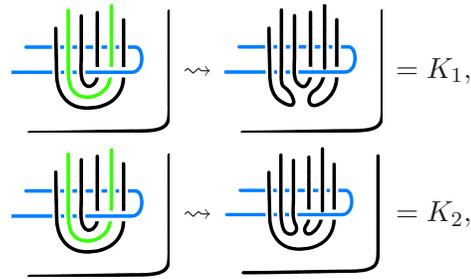
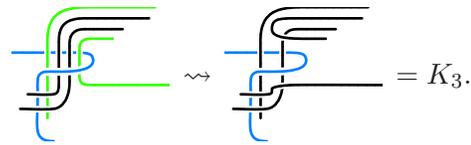
\begin{align*}
\mathfig{0.15}{bands/band-12-pre} & \rightsquigarrow \mathfig{0.15}{bands/band-1-post} = K_1, \\
\mathfig{0.15}{bands/band-12-pre} & \rightsquigarrow \mathfig{0.15}{bands/band-2-post} = K_2, \displaybreak[1] \\
\intertext{and}
\mathfig{0.15}{bands/band-3-pre} & \rightsquigarrow \mathfig{0.15}{bands/band-3-post} = K_3.
\end{align*}

The first two bands take place in the lower right corner of Figure \ref{fig:gompf-1}, while the last takes places near the centre of the diagram. It's relatively easy to see that $K_1$ is actually ribbon; we'll leave this as an exercise to the reader, and spend the rest of this section producing nicer isotopy representatives for the knots $K_2$ and $K_3$.

\subsection{The knot $K_2$}
The knot $K_2$ easily isotopes to
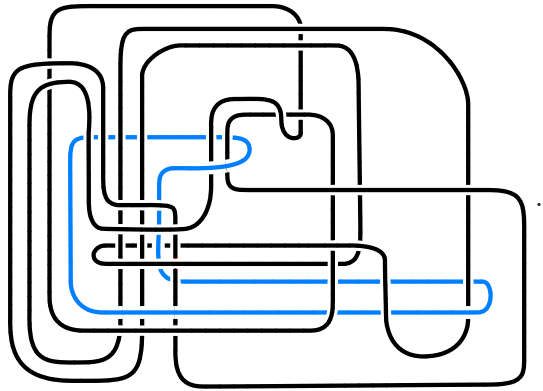
\begin{equation*}
\mathfig{0.5}{bands/A}.
\end{equation*}
We'll show an explicit sequence of isotopies rearranging this knot so that the
strands passing through the twist circle become parallel, performing some simplifications along the way. As before, at each step the isotopy we're
performing is indicated with thick dashed red arcs showing what we're moving, and thin dashed black arcs showing where we're going.
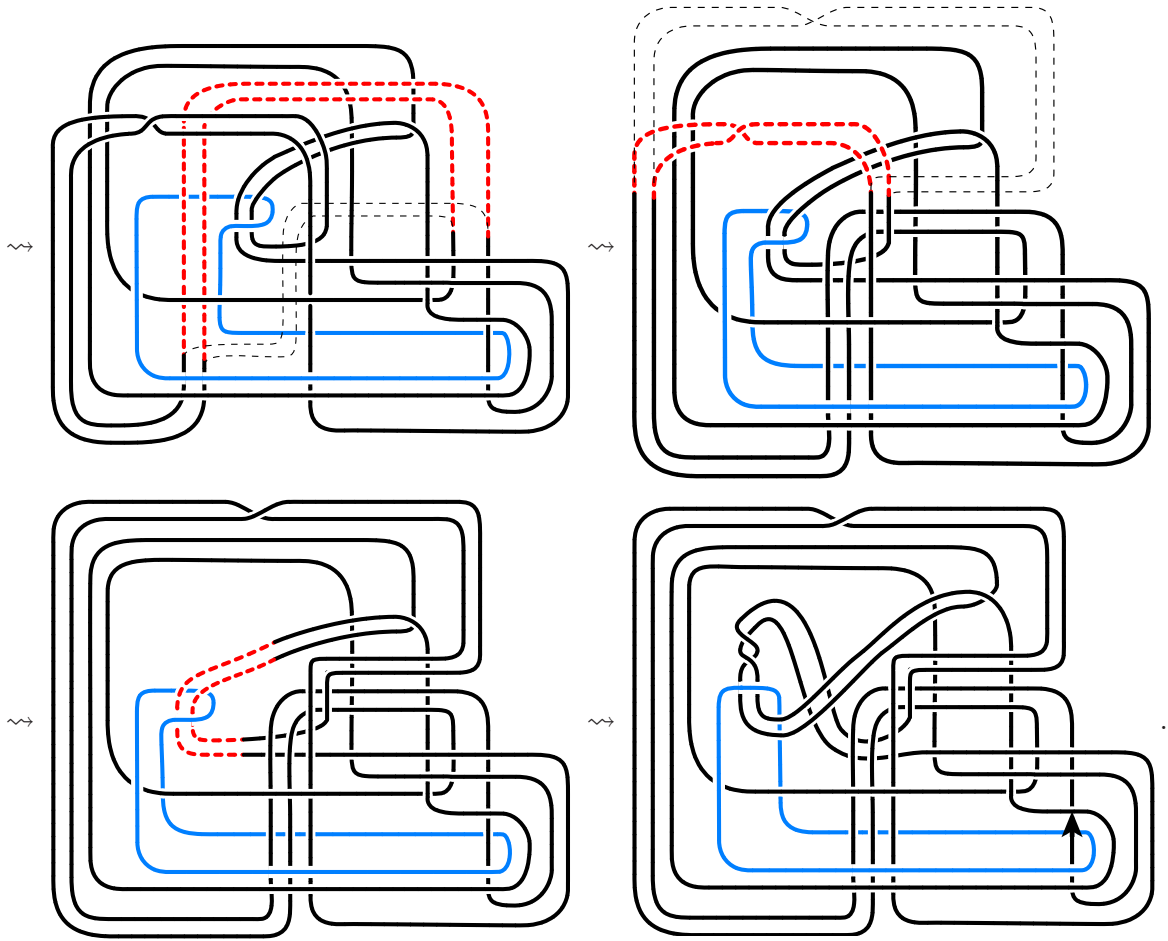
\begin{align*}
\mathfig{0.5}{bands/Ap} & \rightsquigarrow \mathfig{0.5}{bands/B} \displaybreak[1] \\
\mathfig{0.5}{bands/Bp} & \rightsquigarrow \mathfig{0.5}{bands/C} \displaybreak[1] \\
\mathfig{0.5}{bands/Cp} & \rightsquigarrow \mathfig{0.5}{bands/C2} \displaybreak[1] \\
 \rightsquigarrow \mathfig{0.5}{bands/C3}
                        & \rightsquigarrow \mathfig{0.5}{bands/D} \displaybreak[1] \\
 \rightsquigarrow \mathfig{0.5}{bands/E}
                        & \rightsquigarrow \mathfig{0.5}{bands/F}.
\end{align*}

This knot has $86$ crossings, and we've given it the orientation consistent with the orientation of the black component of the link in Figure \ref{fig:link}. Its Gauss code (again, starting from the orientation arrow) is
\begin{quote}
\tt
(-1, -2, -3, 4, 5, -6, -7, 8, 9, -10, -11, -12, -13, 14, 15, 16, 17, -18, 19, -20,
21, -22, -23, 24, 25, 26, 6, 27, -28, -29, -8, -30, -31, 32, -33, -34, 35, 36, 20,
37, 38, 39, -40, -41, -42, -19, -43, 44, 30, 45, -46, 23, 47, 48, 49, 1, 50, -51,
-52, -53, -4, -25, -54, 46, 55, 31, 56, -35, -57, -21, -58, -59, -60, 61, 62, 63,
22, 64, 34, 33, -32, -56, -44, -9, -65, -66, 67, 52, 68, 3, -48, -62, 59, -38, 41,
-69, 70, -15, 12, -71, 72, 73, -74, -75, 66, 28, 76, 7, 77, 54, -24, -47, -63, 58,
-37, 42, -78, 79, -16, 11, 71, 80, -81, -72, -82, 65, 29, -76, -27, 83, 53, -68,
-84, -50, 75, 82, 10, -17, -79, -70, -85, 86, 69, 78, 18, 43, -36, 57, -64, -55,
-45, -77, -26, -5, -83, -67, 51, 84, 2, -49, -61, 60,-39, 40, -86, 85, -14, 13,
-80, 81, -73, 74)
\end{quote}

It's actually easy to do a few more simplifications, reducing the number of crossings to $74$. Mysteriously, however, our programs seemed
to like this presentation more, so we won't bother investigating those simplifications here.

\subsection{The knot $K_3$}
\label{sec:K3}
Alternatively, we can start with the knot $K_3$, and perform the following sequence of isotopies.
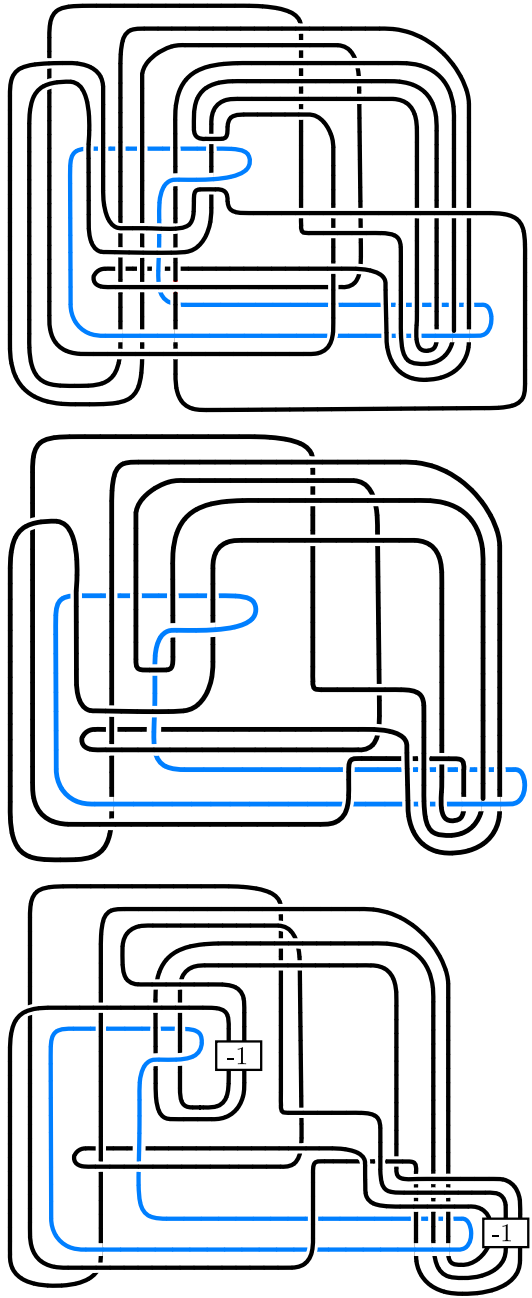
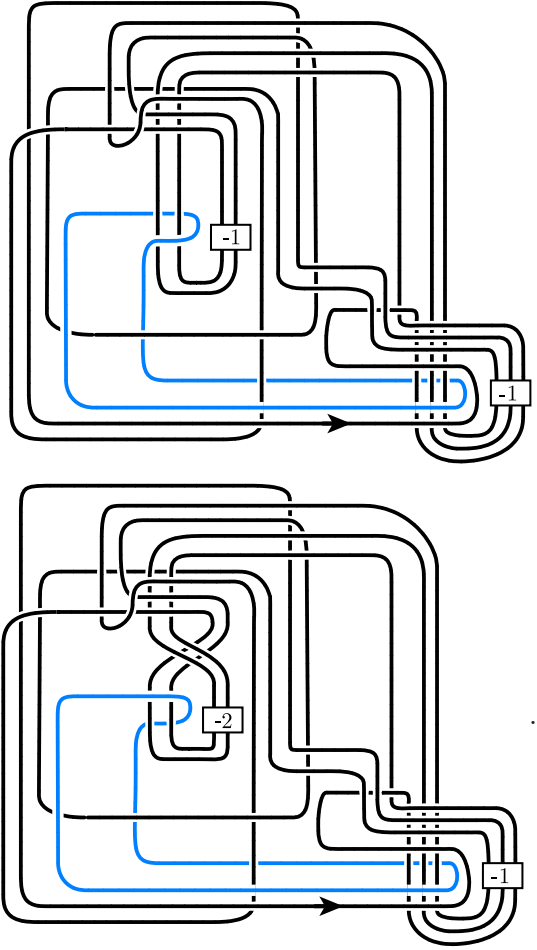
\begin{align*}
\mathfig{0.5}{bands/Z1} \displaybreak[1] \\
\mathfig{0.5}{bands/Z2} \displaybreak[1] \\
\mathfig{0.5}{bands/Z3} \displaybreak[1] \\
\mathfig{0.5}{bands/Z4} \displaybreak[1] \\
\mathfig{0.5}{bands/Z5} \displaybreak[1].
\end{align*}
This knot has $83$ crossings, and its Gauss code is
\begin{quote}
\tt
(1, 2, 3, 4, 5, 6, -7, -8, 9, -10,
-11, -12, -6, 13, -14, 15, -16, 17, 18, -19,
-20, -1, 21, 22, 10, -9, 23, 24, 25, -26,
-27, -28, 29, 30, -31, 32, -33, -34, -17, -35,
-36, 37, 38, 20, 39, 40, -41, -29, 28, 42,
-43, 44, 45, 46, -47, -39, -48, 49, 34, -50,
51, 14, 52, 53, -54, -55, -13, -56, 26, 57,
58, 43, -44, 59, 60, -21, -61, -62, -4, 54,
-53, 63, -64, 36, 65, -66, -37, -3, 62, 67,
12, 7, 68, -25, -69, -70, -59, -45, -40, -71,
72, 33, 50, 16, 73, 64, -63, -74, -15, -51,
56, -68, -75, -23, -76, 77, 70, -58, 78, 27,
-30, -79, 71, 48, 19, 80, 66, -65, -81, -18,
-49, -72, -32, 31, 79, 41, -42, -78, -57, 69,
82, 76, -22, -83, -67, -5, 55, -52, 74, -73,
35, 81, -80, -38, -2, 61, 83, 11, 8, 75,
-24, -82, -77, -60, -46, 47)
\end{quote}

After the initial release of this paper on the arXiv, Nathan Dunfield contacted us with a `better' presentation of $K_3$ (with girth $14$ rather than $16$, see \S \ref{sec:algorithm} below). Its Gauss code is
\begin{quote}
\tt
(-64, -63, -62, 48, 24, 20, 17, -1, -2, 5, 8, 12, -40, -41, -42, -43, 
-44, -45, -46, -47, -48, 84, 91, 90, 88, 82, -76, -75, -74, -73, -10, 
-11, -12, -13, -14, -15, -16, 22, 26, -28, -29, -30, -31, -32, -33, 
40, 49, 68, 73, -91, 85, -80, -79, -78, -77, 72, 67, 61, 51, 42, 39, 
33, 13, 7, 4, 2, 19, -24, -25, -26, -27, 29, 35, 46, 55, 57, 63, 79, 
87, -88, -89, 74, 69, 50, 41, -39, -38, -37, -36, -35, -34, 25, 21, 
-17, -18, -19, -84, -85, -86, -87, 81, 77, 65, 59, 53, 44, 37, 31, 
15, -6, -7, -8, -9, 10, -68, -69, -70, -71, -72, 66, 60, 52, 43, 38, 
32, 14, 6, 3, 1, 18, -20, -21, -22, -23, 30, 36, 45, 54, 58, 64, 78, 
-81, -82, -83, 75, 70, -61, -60, -59, -58, -57, 56, 47, 34, 28, 27, 
23, 16, -3, -4, -5, 9, 11, -49, -50, -51, -52, -53, -54, -55, -56, 
62, 80, 86, -90, 89, 83, 76, 71, -67, -66, -65)
\end{quote}
Although this presentation has $91$ crossings, an easy isotopy reduces it to only $80$.

\subsection{The $m=-1$ case}
\label{sec:minus1}
Finally, for completeness we'll briefly describe the $m=-1$ case. Starting with the diagram from Equation \eqref{eq:Lm-untwisted} we can specialize to $m=-1$, obtaining
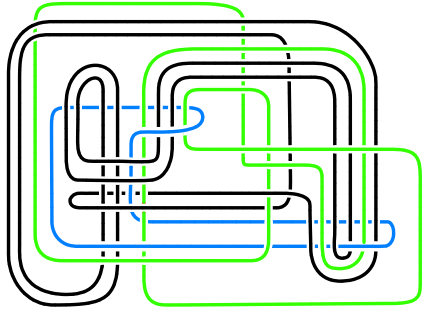
\begin{equation*}
\mathfig{0.4}{minus1}
\end{equation*}
(Note that when $m$ is negative, the strands spiral in the opposite direction.) The entire link still appears to be too hard to calculate with, but we experimented with one band that gives a knot that isn't obviously ribbon. The band is the `same' one as we used to produce $K_2$ in the $m=1$ case, but we won't show the isotopy here; hopefully a Gauss code for the resulting $39$ crossing knot is enough if anyone is interested in trying further computations.
\begin{quote}
\tt
(1, 2, 3, 4, -5, -6, -7, 8, 9, 10, 11, -12, -13, -14, 15, 5, 16, -3, 
-17, 18, 19, 20, 21, 22, 6, 23, -8, -24, -25, -26, 27, 28, 29, 7, 
-23, -15, -30, -31, 32, 33, -34, 13, -10, 25, -28, -21, -35, -1, -36, 
-19, 31, -32, 37, -38, 12, -11, 26, -27, -20, 36, 39, 17, -4, -16, 
-22, -29, 24, -9, 14, 34, 38, -37, -33, 30, -18, -39, -2, 35)
\end{quote}

\section{Calculations}
\label{sec:calculations}
At this point we're ready to go. We have two interesting knots $K_2$ and $K_3$, and we've established that if either is not slice, then the smooth $4$-dimensional Poincar\'{e} conjecture as well as the Andrews-Curtis conjecture must be false. All that remains is to calculate the $s$-invariants, and hope one is nonzero.

But not so fast! We have two obstacles. First, even calculating the two-variable Khovanov polynomial of such a large knot as $K_2$ is a formidable
computational task. Second, in general, calculating the $s$-invariant can be even harder than calculating the polynomial. In the next two sections
we'll address these problems in turn. Section \ref{sec:algorithm} briefly describes the computer program we used. We started with a program
written in Java by Green \cite{green-implementation}, implementing Bar-Natan's algorithm described in \cite{math.GT/0606318}. We then
made a series of improvements, resulting in both significant reductions in memory requirements, and significant improvements in speed. This discussion
will assume some familiarity with Bar-Natan's underlying algorithm. Section \ref{sec:extracting-s}
describes a constraint on the Khovanov two-variable polynomial coming from the $s$-invariant. Under some circumstances the $s$-invariant can determined
directly from the polynomial. However this extraction can itself be a non-trivial computation!

Finally, in Section \ref{sec:results}, we show the output of the program from \S \ref{sec:algorithm} for $K_2$, and apply the the methods of \S \ref{sec:extracting-s} to extract the $s$-invariant.

\subsection{A faster, smaller implementation of Bar-Natan's algorithm}
\label{sec:algorithm}
\newcommand{\KnotTheory}{{\tt KnotTheory`}{}}

The current `state of the art' algorithm for computing Khovanov homology is due to Bar-Natan, and is described in some detail in his paper
\cite{math.GT/0606318}. We give a rather schematic outline of the algorithm here. 

To compute the Khovanov homology of a link $L$, begin by drawing a planar presentation, say with $M$ crossings. (We'll suppose for simplicity that $L$ is not `obviously' split.) Next choose an ordering of the crossings such that each crossing (except the first) is connected to one of the earlier crossings. This lets us construct a sequence of subtangles of $L$, which we'll call $T_m$, so that $T_m$ contains the first $m$ crossings, all the arcs connecting those crossings, none of the later crossings, and none of the arcs connecting those later crossings. In particular, each tangle $T_m$ is just the intersection of the presentation with some disc, and these discs grow monotonically. See Figure \ref{fig:family-of-tangles} for an example. The last tangle $T_M$ is the entire link, so if we can efficiently calculate the Khovanov homology of $T_{m+1}$ from that of $T_m$, we have a chance.

\begin{figure}[!ht]
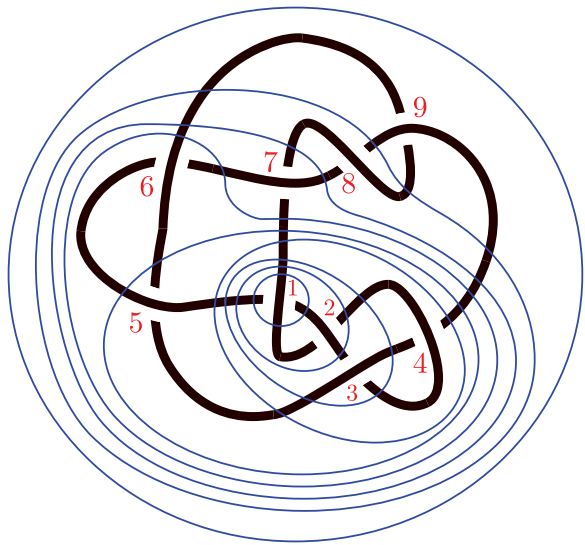

\begin{equation*}
\mathfig{0.55}{algorithm/9_18}
\end{equation*}
\caption{An ordering of the crossings of a presentation of the knot $9_{18}$, and the resulting sequence of increasing subtangles exhausting the diagram. The sequence of girths is $4,4, 6, 4,4,4,4,4,0$.}
\label{fig:family-of-tangles}
\end{figure}

The invariant of tangles defined in \cite{MR2174270} is a complex in a certain category. Just as tangles form a planar algebra (meaning that we can perform arbitrary planar compositions of tangles with appropriate boundaries), \cite{MR2174270} defines a planar algebra structure on these complexes, so that Khovanov homology becomes a map between planar algebras. Thus, given the tangle $T_m$, and the $(m+1)$-st crossing $X$, we can produce the complex $Kh(T_{m+1})$ as the planar composition of $Kh(T_m)$ and the standard complex associated to the crossing, $Kh(X)$. This rule isn't complicated; it's essentially just taking the tensor product of the two complexes.

The real advantage of this scheme comes because in the end, we are only interested in the homotopy type of the complex associated to the whole link. Since the rule for planar composition of tangles is in fact well-defined on the homotopy types of the input tangles, we are free at each step to find a simpler representative of the homotopy type of $Kh(T_m)$. That is, rather than produce an enormous complex $Kh(T_M)$, and then find a simpler homotopy representative, we can perform incremental simplifications at each step along the way. 

What are these simplifications? We just change the complex by \emph{simple homotopies}, that is, we discard contractible direct summands. To identify contractible direct summands, we use the `Gaussian elimination' lemma of \cite{math.GT/0606318}, which shows how, whenever a matrix entry in a differential of the complex is an isomorphism in the underlying category, we can change bases in order to produce an explicit direct summand which is contractible.

The difficulty of a Khovanov homology calculation via this algorithm depends critically on the girth of the link.
The girth of a link presentation $L$ with an ordering of the crossings as above is just the maximum number of boundary points of the intermediate subtangles $T_m$. The girth of a link is the minimum of the girths of its presentations. One expects the girth of a `generic' large knot to scale with the square root of the number of crossings; the knots $K_2$ and $K_3$ have presentations with girth $14$ and $16$ respectively. We had thought it unlikely that there were better presentations, but have been proved wrong by Nathan Dunfield, who found a girth $14$ presentation of $K_3$.
Our rough rule of thumb is that any link with girth $12$ or less is relatively accessible to computer calculations, links of girth $14$ might be possible with sufficient patience and hardware, but that a link of girth $16$ or more is probably impossible. Of course, for a fixed girth more crossings is worse than fewer, but the relationship between memory requirement and girth in actual calculations is striking. A partial heuristic explanation for this is that the indecomposable objects in Bar-Natan's category are just the Temperley-Lieb diagrams; these are counted by Catalan numbers which grow exponentially. We don't expect to be able to perform many simplifications by discarding contractible direct summands, because matrix entry isomorphisms become increasingly uncommon as the variety of source and target objects increases.

The existing implementations of this algorithm use the version of Khovanov homology
described in \cite{MR2174270}, for which the chain complex associated to a tangle lives in a category of tangle smoothings and cobordisms between them, modulo certain relations  (this category is a categorification of the Temperley-Lieb category). 
It should
also be possible to use the more algebraic version described in \cite{MR1928174}, for which the chain complexes live in categories of bimodules over certain rings. To our knowledge, however, no such implementation exists.\footnote{A defect of this version of Khovanov homology for tangles is that
it gives operations for stacking tangles in two directions (via `external' tensor product of bimodules, and tensoring bimodules over the rings), rather than for arbitrary planar compositions of tangles. Of course, arbitrary planar compositions can be decomposed into sequences of stacking operations, but nevertheless this would complicate the algorithm as described in \cite{math.GT/0606318} and here.}

Two independent implementations of this algorithm exist to date. The first, written by Bar-Natan in \MMA, is available as
part of the \KnotTheory package, from \url{http://katlas.org}. For essentially all purposes, however, it has been made obsolete by Jeremy Green's
java based implementation which is also available through the \KnotTheory package.
These implementations will be referred to as \code{FastKh} and \code{JavaKh} respectively; these are also the names used within the \KnotTheory package.
Bar-Natan's implementation was solely intended as a demonstration of the algorithm, and no significant attempts were made to
optimise the program for either speed or memory consumption. Green's implementation is on the order of thousands of times faster than Bar-Natan's. 

Our implementation is an update of \code{JavaKh}. To distinguish the the original from the updated version, we'll use the names \code{JavaKh-v1} and \code{JavaKh-v2}. Some code (particularly that dealing with `cobordism arithmetic') remains unchanged, but most of the
`outer layers' have been rewritten. This update has already been used by other researchers, in particular in published work in \cite{0901.4039}. The changes we made fall into four categories described below.

\paragraph{Interface improvements.} Large calculations can now give progress reports, at various levels of detail. At its most verbose, every matrix entry isomorphism which is discarded gets reported, along with elapsed time, memory use, and the capacities of various internals caches. These reports are not available through the \MMA interface to \code{JavaKh}, but only through the direct command-line interface. (See \S \ref{sec:running-javakh}, or just try the switches {\tt -i} and {\tt -d}.)
\paragraph{Memory optimizations.} A significant number of memory-saving tweaks have been made throughout the code, for example
\begin{itemize}
\item Using arrays of {\tt byte}s, rather than arrays of {\tt int}s, to store topological descriptions of surfaces. (It would be possible to go much further here, as at least up to girth $14$ one could package some of this data into arrays of `half-bytes'. We had insufficient enthusiasm for writing this sort of bit-flipping code.)
\item Using linked lists or hash-maps instead of pre-allocated arrays for matrix entries or terms of linear combinations. (More generally, we've made it much easier to `drop in' a different implementation of a particular storage model, and tried benchmarking a few different options.)
\item Storing each complex on disk, instead of in memory, and only loading a few relevant homological heights at a time. This feature is not enabled by default, as it signficantly slows computations. It can be enabled on the command line with the switch {\tt -C}. This is not as effective as we'd at first hoped; the memory usage is sharply peaked in the middle of the complex, and so a significant fraction of the complex must be loaded in memory even to deal with three consecutive differentials, as required by the Gaussian elimination step of the algorithm. 
This was an attempt to work around Java's infamous unwillingness to use virtual memory directly. We suspect that even if virtual memory were available, the Gaussian elimination algorithm would be extremely slow if the entire matrix for a differential could not be held in memory.
\item Caching small arrays of {\tt byte}s and {\tt int}s. While performing cobordism arithmetic, many redundant copies of small arrays with small integer entries are generated. In some circumstances, passing these through a cache upon creation results in significant memory savings. (At the same time, we disabled caching routines in the original ``JavaKh'' implementation which worked at the level of cobordisms.)
\end{itemize}
\paragraph{Allowing arbitrary orderings of crossings.} Notice in the schematic description of the algorithm above that we need to choose the order in which we add the crossings to the `inner' tangle. In both  \code{FastKh} and the original \code{JavaKh-v1}, the first crossing is chosen essentially arbitrarily (whatever comes first in the {\tt PD} presentation produced by \MMA), and at each step the next crossing is chosen from amongst those which are `maximally connected' to the current inner tangle. This should be thought of as a greedy algorithm attempting to minimise the maximal girth of the intermediate inner tangles (that is, the number of boundary points). Our update of \code{JavaKh} allows the user to disable this algorithm, and just process the crossings in the order that they appear in the presentation of the link. We've also written an auxiliary program in \MMA that attempts to order the crossings so as to minimise maximal girth. This program isn't particularly clever; it essentially uses the greedy algorithm described above, but when there are alternatives (e.g. for the first crossing, or for a subsequent crossing from amongst those which are equally maximally connected to the inner tangle) it makes random choices, and tries many times. This is available via the function {\tt FindSmallGirthOrdering} in the package \KnotTheory, and can be chained with the \code{JavaKh} algorithm, for example by the commands
\begin{align*}
	\tt Kh[FindSmallGirthOrd&\tt ering[TorusKnot[7,6]], \\
	                                                & \tt ExpansionOrder \to False][q,t] \\
\intertext{or}
	\tt Kh[FindSmallGirthOrd&\tt ering[ TorusKnot[7,6], 1000],  \\
	                                                 & \tt ExpansionOrder \to False][q,t]
\end{align*}
to specify that {\tt FindSmallGirthOrdering} should return the best candidate after 1000 trials.

\paragraph{Canceling blocks of isomorphisms.}
Bar-Natan's original algorithm looks for a matrix entry in the differential which is an isomorphism (that is, a multiple of a cylinder cobordism), and performs a change of basis so this matrix entry becomes a contractible direct summand, which is then discarded. This step must be repeated many times as long as more isomorphisms can be found, and for large tangles this becomes extremely time-consuming. In \code{JavaKh-v2}, we instead look for submatrices of a differential which are diagonal, with all diagonal entries isomorphisms, and discard the corresponding contractible direct summands. While most of our other improvements concentrated on optimising memory consumption, this modification results in a significant speedup in most cases. 
It might be possible to go further in this direction, for example by looking for upper triangular submatrices all of whose diagonal entries are isomorphisms. Computing the inverse of such matrices is very efficient, but the cost of looking for such matrices rather than just diagonal ones might limit the improvements available.

\subsubsection{Running \code{JavaKh}}
\label{sec:running-javakh}
You have essentially two options for running \code{JavaKh}; via the \code{KnotTheory`} package in \MMA, which is convenient but does not offer access to all functionality, or directly from the command line, which is much more suitable for long computations in which progress reports are useful. The version of \code{JavaKh-v2} which was current at the time of writing is included in the \code{arXiv} sources for this article, but it's likely that the version included in the \code{KnotTheory`} package is more up-to-date. You can always find the latest stable version at \url{http://katlas.org/svn/KnotTheory/tags/stable/KnotTheory/JavaKh-v2}.

As long as you have a recent version of the \code{KnotTheory`} package, the function \code{Kh} automatically uses the \code{JavaKh-v2} implementation. This can be modified using the \code{Program} option; the details of this and other options are summarised in Figure \ref{fig:mma}
\begin{figure}[!htb]
\begin{tabular}{|c|c|p{3in}|}
\hline
option & value & description \\
\hline\hline
\code{Program} & \code{"FastKh"} & Use the original \MMA implementation. \\
                          & \code{"JavaKh-v1"} & Use Green's \code{Java} implementation. \\
                          & \code{"JavaKh-v2"} & (default) Use the modified implementation described here. \\
\hline
\code{JavaOptions} & e.g. \code{"-Xmx512m"} & Arguments to pass to the \code{Java} virtual machine.  This example allows the heap size to grow to \code{512} megabytes. Depending on your hardware and operating system, increasing this parameter may allow computations of larger knots. \\
\hline
\code{ExpansionOrder} & \code{Automatic} & (default) Automatically reorder crossings using an internal greedy algorithm. \\
                                      & \code{False} & Do not reorder crossings (ignored by \code{JavaKh-v1}).\\
\code{Modulus} & 0 & (default) Work over the integers. \\
                          & $p$ & Work over the integers mod $p$. \\
\code{Universal} &\code{True} & Work in `universal mode'. Undocumented, and not for the faint-hearted, but see also the \code{UniversalKh} function. \\
\hline
\end{tabular}
\caption{The options available for the function \code{Kh} in the \code{KnotTheory`} package.}
\label{fig:mma}
\end{figure}

Thus for example we might run
\begin{align*}
\tt Kh[Knot[8,19]& \tt, Program \to "FastKh", Modulus \to 5] \\
\intertext{to compute the mod $5$ homology using the original implementation, or}
\tt Kh[largeknot& \tt, JavaOptions \to "-Xmx2000m", ExpansionOrder \to False]
\end{align*}
to compute the homology of a diagram in which we've already chosen a good ordering of the crossings, and allowing \code{Java} to use up to $2$ gigabytes of physical memory.

You can also run \code{JavaKh-v2} from the command line, and this is more suitable for long calculations. You'll need to tell \code{Java} where to find the class and library files (everything in the \code{bin} directory), and to execute the class \code{org.katlas.JavaKh.JavaKh}. The available command line options appear in Figure \ref{fig:commandline}.

\begin{figure}
\begin{tabular}{|c|p{4in}|}
\hline
option & description \\
\hline\hline
\code{-i} & provide more verbose output, including progress reports. \\
\hline
\code{-Z} & work over integers. \\
\code{-Q} & work over the rationals. \\
\code{-m <prime>} & work over the integers mod $p$. \\
\code{-U} & run in `universal mode'.\\
\hline
\code{-O} & don't reorder crossings internally. \\
\code{-C} & save intermediate results in files in the current directory, and/or resume from such files (just use \code{ctrl-C} to break out of the computation). \\
\code{-D} & (experimental!) switch to a much slower memory-saving mode (you should also specify \code{-Djava.io.tmpdir=\$TMPDIR} on the command line).\\
\code{-P} & (likely to crash!) use multiple CPUs.\\
\hline
\end{tabular}
\caption{Command-line options available for \code{JavaKh-v2}.}
\label{fig:commandline}
\end{figure}

\begin{example}
This assumes you have a \code{UNIX}-like environment, and a copy of the \code{JavaKh-v2} files (taken from the \package{KnotTheory} package, for example) in the directory \code{$\sim$/JavaKh-v2}.
\begin{verbatim}
JAVAKHHOME=~/JavaKh-v2
CLASSPATH=$JAVAKHHOME/jars/commons-cli-1.0.jar:
        $JAVAKHHOME/jars/commons-logging-1.1.jar:
        $JAVAKHHOME/jars/commons-io-1.2.jar:
        $JAVAKHHOME/jars/log4j-1.2.12.jar:
        $JAVAKHHOME/bin
java -Xmx28000m -classpath $CLASSPATH org.katlas.JavaKh.JavaKh
        -O --mod 13 -i -C < pd
\end{verbatim}
The first two lines just prepare the \code{Java} classpath. The final line will compute, using \code{28gb} of physical RAM, the Khovanov homology of the diagram in the file \code{pd} (in the \code{PD} notation used by the \package{KnotTheory} package), without reordering any crossings, in the integers mod $13$, giving verbose output and storing intermediate results to disk.
\end{example}

%

\subsection{Extracting the $s$-invariant from a Khovanov polynomial}
\label{sec:extracting-s}
\begin{thm}
\label{thm:decomposition}%
For any knot $K$, there's an integer $s$ and a family of two variable Laurent polynomials $f_k \in \Natural[q^{\pm1}, t^{\pm1}]$ for $k \geq 2$, so that
\begin{equation*}
\Kh{K}{q,t} = q^s (q + q^{-1}) + \sum_{k\geq2} f_k(q,t)(1+q^{2k}t).
\end{equation*}
(Clearly only finitely many of the $f_k$ are nonzero, since they have non-negative coefficients.)
The integer $s$ is the $s$-invariant of the knot $K$.
\end{thm}

Although this theorem hasn't appeared in the literature in this form, it follows from the discussion at the end of \cite{MR2232858}, or by thinking about the invariant of a cut-open knot in the variation of Bar-Natan's formalism \cite{MR2174270} for which the genus $3$ surface is a formal parameter. The function \code{UniversalKh} in the \code{KnotTheory`} package computes the invariants $f_k$. For us,
the point of this theorem is that it's often (and perhaps always) possible to extract the $s$-invariant of a knot knowing nothing more than the
graded dimensions of the Khovanov homology. That is, given the two-variable polynomial, we can often show that for all possible decompositions of the form in the theorem, the same value of $s$ appears.

\begin{conj}
\label{conj:only-f2}
In fact, for any knot $K$, only the polynomial $f_2$ is nonzero.
\end{conj}
\begin{rem}
This \emph{isn't} the same as \cite[Conjecture 3.9]{MR2034399}, which was proved in \cite{MR2232858}.

Extensive computations by Shumakovitch \cite{math.GT/0405474, MR2384833}, the authors and others support this conjecture, and it also holds for our examples (see below). It's easy to see that decompositions with only $f_2$ nonzero are unique.
\end{rem}

An amusing and straightforward corollary of this conjecture is an easy formula for the $s$-invariant:
\begin{equation*}
q^{s(K)} = \frac{\Kh{K}{q, -q^{-4}}}{q+q^{-1}}
\end{equation*}


Finding all possible decompositions as in Theorem \ref{thm:decomposition} can become quite difficult! The smallest example we know of where the decomposition
is not unique is the $(7,6)$ torus knot.\footnote{It's unique for all knots with at most 14 crossings, for example.} Here there are four possible decompositions (all giving the same value of the $s$-invariant):
\begin{align*}
\operatorname{Kh}(T(7,6))(q,t)
    &  = (q+q^{-1}) + (p_{\star,4}(q,t) + p_{A,4}(q,t))(1 + q^4 t) + p_{A,6}(q,t)(1+q^6t) \\
    & = (q+q^{-1}) + (p_{\star,4}(q,t) + p_{B,4}(q,t))(1 + q^4 t) + p_{B,6}(q,t)(1+q^6t) \\
    & = (q+q^{-1}) + (p_{\star,4}(q,t) + p_{C,4}(q,t))(1 + q^4 t) + p_{C,6}(q,t)(1+q^6t) \\
    & = (q+q^{-1}) + (p_{\star,4}(q,t) + p_{D,4}(q,t))(1 + q^4 t) + p_{D,6}(q,t)(1+q^6t)
\end{align*}
where
\begin{align*}
p_{A,4}(q,t) & = t^{15} q^{53}+t^{14} q^{47}+t^{11} q^{47}+t^{10} q^{41} & p_{A,6}(q,t) & = 0 \\
p_{B,4}(q,t) & = t^{15} q^{53}+t^{14} q^{47} & p_{B,6}(q,t) & = t^{11} q^{45} + t^{10} q^{41} \\
p_{C,4}(q,t) & = t^{11} q^{47}+t^{10} q^{41} & p_{C,6}(q,t) & = t^{15} q^{51} + t^{14} q^{47} \\
p_{D,4}(q,t) & = 0 & p_{D,6}(q,t) & = t^{15} q^{51} + t^{14} q^{47} + t^{11} q^{45} + t^{10} q^{41} \\
\end{align*}
and
\begin{align*}
p_{\star,4}(q,t) & =
    t^{18} q^{53}+t^{16} q^{51}+t^{15} q^{51}+t^{16} q^{49}+t^{14} q^{49}+t^{13} q^{49}+t^{12} q^{47}+2 t^{12} q^{45} \\
 &  +2 t^{10} q^{43}+2 t^8 q^{41}+t^8 q^{39}+t^6 q^{39}+t^6 q^{37}+t^4 q^{37}+t^4 q^{35}+t^2 q^{33}
\end{align*}

Note that for all decompositions we have $s=0$, and that the first decomposition is consistent with the conjecture. We don't
know of any examples where different values of $s$ occur in different decompositions, although this
is certainly possible for arbitrary polynomials in $\Natural[q^{\pm1},t^{\pm 1}]$. For example
\begin{align*}
q^3 + q + q^{-1} + q^7 t 	& = q^0 (q + q^{-1}) +  q^3 (1+q^4 t) \\
					& = q^2 (q+ q^{-1}) + q^{-1} (1+q^8 t),
\end{align*}
has decompositions with either $s=0$ or $s=2$.

For the knots we're interested in, we need a trick to make the task of finding all decompositions manageable. We'll use the following one.

We'll say a Laurent polynomial $z(q,t)$ has a $1$-decomposition if it can be written as in Theorem \ref{thm:decomposition}, and it has a $0$-decomposition
if it can be written that way, but without the initial $q^s(q+q^{-1})$ term.

Clearly, if $u(q,t)$ has a $1$-decomposition and $v(q,t)$ has a $0$-decomposition, then $u(q,t)+v(q,t)$ has a $1$-decomposition.
\begin{lem}
\label{lem:cutting-decompositions}
Any $1$-decomposition for $\Kh{K}{q,t} = \sum_{j,r} a_{j,r} q^j t^r$ arises in this way from a $1$-decomposition for $u(q,t)$, and a $0$-decomposition for $v(q,t)$, where
\begin{align*}
u(q,t) & = \sum_{\substack{j \in \Integer\\r\geq 0}} a_{j,r} q^j t^r + t^{-1} w(q) \\
v(q,t) & = \sum_{\substack{j \in \Integer\\r < 0}} a_{j,r} q^j t^r - t^{-1} w(q)
\end{align*}
for some Laurent polynomial $w\in \Natural[q^{\pm 1}]$, so $v(q,t) \in t^{-1}\Natural[q^{\pm1},t^{-1}]$.
\end{lem}
\begin{rem}
You could think of this as `cutting $\Kh{K}{q,t}$ into two pieces, along the $t^{-1}$ line',
and the lemma as a statement about `fibered products' of decompositions.
We could cut elsewhere\footnote{Replace the inequalities in the
index of the summation with $r \geq k \geq 0$ and $r < k$ for $u$ and $v$ respectively, and similarly the seconds terms with $+t^{k-1} w(q)$ and $-t^{k-1} w(q)$.}, but we'll only use this case.
The proof is easy; just observe that the polynomials $(1+q^{2k} t)$ only span $2$ different $t$ degrees.
\end{rem}

\subsection{Results}
\label{sec:results}
Computing the two-variable polynomial for $K_2$ took approximately $4$ weeks on a dual core AMD Opteron 285 with 32 gb of RAM. At this point, we haven't been able to do the calculation for $K_3$. With the current version of the program, after about two weeks the program runs out of memory and aborts.

Here's the polynomial for $K_2$:

\begin{align}
\operatorname{Kh}(K_2)&(q,t) =
        q^{-45} t^{-32} + q^{-41} t^{-31} + q^{-39} t^{-29} + q^{-35} t^{-28} + q^{-37} t^{-27} + q^{-37} t^{-26} \notag \\
    & + q^{-33} t^{-26} + q^{-35} t^{-25} + q^{-33} t^{-25} + q^{-35} t^{-24} + 2 q^{-31} t^{-24} + q^{-33} t^{-23} \notag \\
    & + 2 q^{-31} t^{-23} + q^{-27} t^{-23} + q^{-33} t^{-22} + 2 q^{-29} t^{-22} + q^{-27} t^{-22} + q^{-31} t^{-21} \notag \displaybreak[1]  \\
    & + 3 q^{-29} t^{-21} + q^{-25} t^{-21} + q^{-31} t^{-20} + 3 q^{-27} t^{-20} + 2 q^{-25} t^{-20} + 4 q^{-27} t^{-19} \notag \\
    & + 2 q^{-23} t^{-19} + q^{-27} t^{-18} + 2 q^{-25} t^{-18} + 4 q^{-23} t^{-18} + 4 q^{-25} t^{-17} + q^{-23} t^{-17} \notag \\
    & + 3 q^{-21} t^{-17} + q^{-19} t^{-17} + 4 q^{-25} t^{-16} + 2 q^{-23} t^{-16} + 6 q^{-21} t^{-16} + q^{-17} t^{-16} \notag \displaybreak[1]  \\
    & + 4 q^{-23} t^{-15} + 5 q^{-21} t^{-15} + 3 q^{-19} t^{-15} + 2 q^{-17} t^{-15} + q^{-23} t^{-14} + q^{-21} t^{-14} \notag \\
    & + 8 q^{-19} t^{-14} + q^{-17} t^{-14} + q^{-15} t^{-14} + 3 q^{-21} t^{-13} + 6q^{-19} t^{-13} + 3 q^{-17} t^{-13} \notag \\
    & + 4 q^{-15} t^{-13} + q^{-21}t^{-12} + 2 q^{-19} t^{-12} + 9 q^{-17} t^{-12} + 5 q^{-15} t^{-12} +2 q^{-13} t^{-12} \notag \displaybreak[1]  \\
    & + 7 q^{-17} t^{-11} + 4 q^{-15} t^{-11} + 7 q^{-13}t^{-11} + 3 q^{-17} t^{-10} + 7 q^{-15} t^{-10} + 7 q^{-13} t^{-10} \notag \\
    & +2 q^{-11} t^{-10} + q^{-9} t^{-10} + 8 q^{-15} t^{-9} + 6 q^{-13}t^{-9} + 9 q^{-11} t^{-9} + q^{-9} t^{-9} \notag \\
    & + 3 q^{-15} t^{-8} + 5q^{-13} t^{-8} + 13 q^{-11} t^{-8} + 4 q^{-9} t^{-8} + 2 q^{-7}t^{-8} + 5 q^{-13} t^{-7} \notag \displaybreak[1]  \\
    & + 8 q^{-11} t^{-7} + 9 q^{-9} t^{-7} + 5q^{-7} t^{-7} + q^{-5} t^{-7} + 5 q^{-11} t^{-6} + 13 q^{-9} t^{-6} \notag \\
    & +6 q^{-7} t^{-6} + 4 q^{-5} t^{-6} + q^{-11} t^{-5} + 8 q^{-9} t^{-5}+ 11 q^{-7} t^{-5} + 8 q^{-5} t^{-5} \notag \\
    & + q^{-3} t^{-5} + 2 q^{-9}t^{-4} + 12 q^{-7} t^{-4} + 10 q^{-5} t^{-4} + 6 q^{-3} t^{-4} + 7q^{-7} t^{-3} \notag \displaybreak[1]  \\
    & + 9 q^{-5} t^{-3} + 12 q^{-3} t^{-3} + 2 q^{-1} t^{-3}+ 9 q^{-5} t^{-2} + 12 q^{-3} t^{-2} + 8 q^{-1} t^{-2} \notag \\
    & + q^{1} t^{-2}+ 3 q^{-5} t^{-1} + 7 q^{-3} t^{-1} + 15 q^{-1} t^{-1} + 5 q^{1}t^{-1} + q^{3} t^{-1} +3 q^{-3} t^{0} \notag \\
    & + 14 q^{-1} t^{0} + 10 q^{1} t^{0} + 6 q^{3} t^{0} +q^{-3} t^{1} + 5 q^{-1} t^{1} + 11 q^{1}t^{1} + 10 q^{3} t^{1} \notag \displaybreak[1]  \\
    & + 2 q^{5} t^{1} + q^{-1} t^{2} + 8 q^{1} t^{2}+ 10 q^{3} t^{2} + 8 q^{5} t^{2} + 2 q^{1} t^{3} + 7 q^{3} t^{3} + 10q^{5} t^{3} \notag \\
    & + 5 q^{7} t^{3} + 4 q^{3} t^{4} + 7 q^{5} t^{4} + 6 q^{7}t^{4} + 3 q^{9} t^{4} + q^{3} t^{5} +5 q^{9} t^{5} + 2 q^{5} t^{6} + 5 q^{7} t^{6} \notag \\
    &  + 7 q^{9} t^{6} + 4q^{11} t^{6} + 4 q^{5} t^{5} + 8 q^{7} t^{5} + q^{7} t^{7} + 5 q^{9} t^{7} + 4 q^{11} t^{7} + 3q^{13} t^{7} \notag \displaybreak[1]  \\
    &  + 2 q^{9} t^{8} + 4 q^{11} t^{8} + 3 q^{13} t^{8} + 3q^{11} t^{9} + 4 q^{13} t^{9} + 3 q^{15} t^{9} + q^{11} t^{10} +q^{13} t^{10} \notag \\
    &  + 3 q^{15} t^{10} + 2 q^{17} t^{10} + q^{13} t^{11} + 2q^{15} t^{11} + q^{17} t^{11} + q^{13} t^{12} + 2 q^{17} t^{12} +q^{19} t^{12} \notag \\
    &  + 2 q^{17} t^{13} + q^{21} t^{13} + q^{17} t^{14} +q^{19} t^{14} + q^{21} t^{14} + q^{19} t^{15} + q^{21} t^{15} +q^{23} t^{15} \notag \\
    &  + q^{23} t^{16} + q^{23} t^{17} + q^{27} t^{18} \label{eq:KhK2} 
\end{align}

The coefficients of this polynomial are shown in tabular form in Figure \ref{fig:coefficients}.

We'll now apply Lemma \ref{lem:cutting-decompositions} to extract the $s$-invariant. We first observe that the coefficient of $t^{-1}$ in $\Kh{K_2}{q,t}$
is $3 q^{-5} + 7 q^{-3} + 15 q^{-1} + 5 q^{1} + q^{3}$. The polynomial $w(q)$ in the lemma must have coefficients no greater than these.
In fact, it must have terms $3 q^{-5} + 7 q^{-3}$; it's easy to see from Figure \ref{fig:coefficients} that these terms can not be part of
any $0$-decomposition of $v(q,t)$. Moreover, the terms $5 q^{1} + q^{3}$ can not be part of $w(q)$, since they can not be part of any $1$-decomposition
of $u(q,t)$. Thus we need only consider $w(q) = 3 q^{-5} + 7 q^{-3} + k q^{-1}$ for some $0 \leq k \leq 15$. Happily, finding $1$-decompositions
of the resulting $u(q,t)$ is easily tractable by computer, or painful-but-tractable by hand. We find that there are exactly $30$ such decompositions,
all with $k=6$. In all $30$ cases, we have $s=0$, and so, rather sadly, this must also be the $s$-invariant of the knot $K_2$.

Finding $0$-decompositions of $v(q,t)$ seems to be intractable directly, so we can't even tell you how many decompositions there are of the whole
polynomial! Presumably one could `cut' $v(q,t)$ and apply the Lemma again. Nevertheless, it's easy to see that there's a decomposition satisfying the conjecture \ref{conj:only-f2};
just divide $\Kh{K_2}{q,t} - (q+q^{-1})$ by $(1+q^4 t)$.

\begin{figure}[ht]
\begin{center}
\resizebox{\textwidth}{!}{
\begin{tabular}{|c||c|c|c|c|c|c|c|c|c|c|c|c|c|c|c|c|c|c|c|c|c|c|c|c|c|c|c|c|c|c|c|c|c|c|c|c|c|c|c|c|c|c|c|c|c|c|c|c|c|c|c|c|c|c|c|c|c|c|c|c|c|c|c|c|c|c|}
\hline
\input{sections/table-K2}
\end{tabular}
}
\end{center}
\caption{The coefficient of $q^{\ell+2r} t^r$ in $\Kh{K_2}{q,t}$.}
\label{fig:coefficients}
\end{figure}

Finally, the knot considered in \S\ref{sec:minus1} coming from the $m=-1$ case is small enough that we can use the function \code{UniversalKh} to compute the $s$-invariant directly. We find that
\begin{align*}
Kh(K^{m=-1})(q,t) &=q^{-23} t^{-14}+q^{-19} t^{-13}+q^{-19} t^{-12}+q^{-17} t^{-11}+q^{-15} t^{-11}+\displaybreak[1]\\&q^{-17} t^{-10}+q^{-13} t^{-10}+2 q^{-13} t^{-9}+q^{-15} t^{-8}+q^{-9} t^{-8}+2 q^{-11} t^{-7}+\displaybreak[1]\\&q^{-13} t^{-6}+q^{-11} t^{-6}+q^{-7} t^{-6}+3 q^{-9} t^{-5}+q^{-7} t^{-5}+2 q^{-9} t^{-4}+\displaybreak[1]\\&2 q^{-5} t^{-4}+2 q^{-7} t^{-3}+3 q^{-5} t^{-3}+q^{-7} t^{-2}+q^{-5} t^{-2}+3 q^{-3} t^{-2}+\displaybreak[1]\\&q^{-1} t^{-2}+4 q^{-3} t^{-1}+q^{-1} t^{-1}+t^{-1} q+q^{-3}+2 q^{-1}+4 q+q^{-1} t+\displaybreak[1]\\&2 q t+q^{3} t+2 q^{3} t^{2}+q^{5} t^{2}+q^{3} t^{3}+q^{7} t^{3}+q^{7} t^{4} \\
\intertext{with}
s & = 0 \\
f_2 & = q^{-23} t^{-14}+q^{-19} t^{-12}+q^{-17} t^{-11}+q^{-17} 
t^{-10}+q^{-13} t^{-9}+\displaybreak[1]\\&q^{-15} t^{-8}+q^{-11} t^{-7}+q^{-13} 
t^{-6}+q^{-11} t^{-6}+2 q^{-9} t^{-5}+2 q^{-9} t^{-4}+\displaybreak[1]\\&2 q^{-7} 
t^{-3}+q^{-5} t^{-3}+q^{-7} t^{-2}+q^{-5} t^{-2}+q^{-3} t^{-2}+3 
q^{-3} t^{-1}+q^{-3}+\\&q^{-1}+q^{-1} t+q t+q^{3} t^{2}+q^{3} t^{3}
\end{align*}
and all other $f_k = 0$.

\subsection{Hyperbolic volume and homology}
In \cite{Dunfield-volumes} Dunfield noticed a correlation between the hyperbolic volume of a knot and the absolute value of its determinant. Unfortunately it doesn't work nearly as well for non-alternating knots as for alternating knots. In \cite{MR2034399}, based on the first computations of Khovanov homology available, Khovanov noticed that the correlation is even better between the hyperbolic value and the rank of the homology groups. Indeed, for alternating knots it is now known that $\abs{\det{K}} = \rank(\KhH{K}) - 1$, and that $\abs{\det{K}} \leq \rank(\KhH{K}) - 1$ for all knots. These correlations were noticed by looking at knots up to either $13$ (for the determinant) or $11$ (for the rank) crossings. Having just completed calculating the Khovanov homology calculation of a relatively huge knot, it's interesting to check this correlation. Unfortunately, it appears to fail. Substituting $q=1$, $t=1$ into $\Kh{K_2}{q,t}$ from Equation \eqref{eq:KhK2} shows the rank is $650$, and a calculation using Dunfield's {\tt{SnapPeaPython}} shows the hyperbolic volume is approximately $17.2879...$. Figure \ref{fig:volumes} shows a plot comparing $\log(\rank(\KhH{K}))$ and $\operatorname{volume}(K)$ for a random sample of non-alternating knots of $11$, $12$, $13$ and $14$ crossings (blue, red, green and black points, respectively). The extra point, well outside the obvious cluster, is the corresponding data for $K_2$. Perhaps the observed correlation is an artifact of sampling knots by crossing number?

\begin{figure}[!ht]
\labellist \Large\hair 2pt

  \pinlabel $K_2$ at 236 200
  \pinlabel $\operatorname{volume}(K)$ at 190 8
  \pinlabel \rotatebox{90}{$\log(\rank(\KhH{K}))$} at 2 130

\endlabellist
\centering
\includegraphics[scale=1.2]{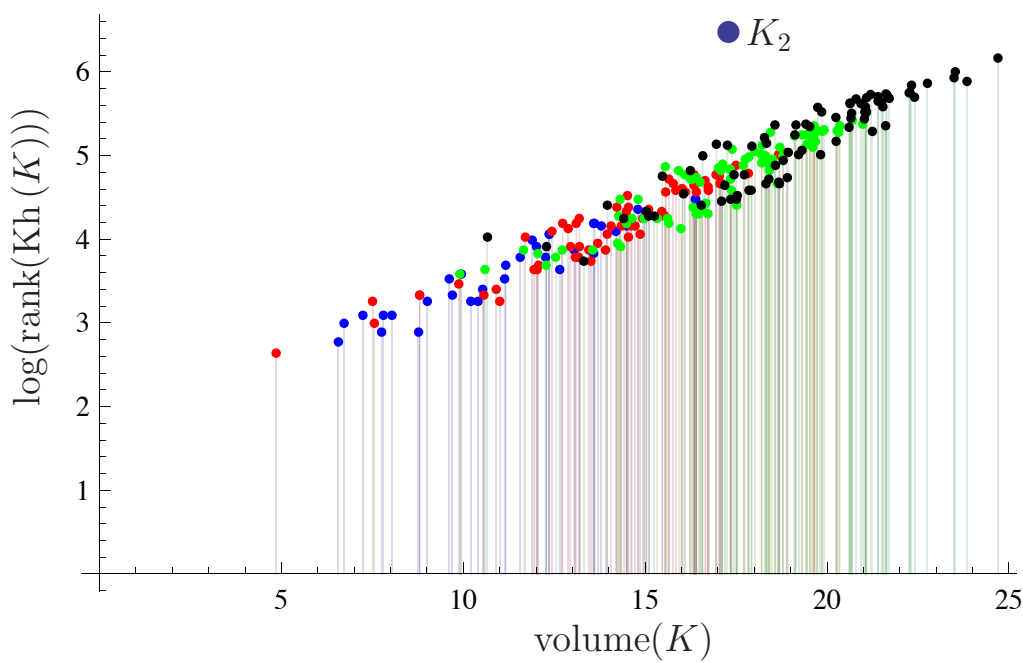}
\caption{Volumes and ranks of homology groups for a sample of $11$-$14$ crossing knots (shown as blue, red, green and black points), and also for $K_2$.}\label{fig:volumes}
\end{figure}

\newcommand{\urlprefix}{}
\bibliographystyle{gtart}
\bibliography{bibliography/bibliography}

This paper is available online at \arxiv{0906.5177}, and at
\url{http://tqft.net/SPC4}.

\end{document}